\documentclass[12pt]{amsart}
\usepackage{amssymb}
\usepackage{diagrams}\diagramstyle[tight,centredisplay,dpi=600,noPostScript,heads=LaTeX]%
\newtheorem{theorem}{Theorem}

\newtheorem{corollary}[theorem]{Corollary}
\newtheorem{sublemma}{Lemma}[theorem]
\newtheorem{lemma}[theorem]{Lemma}

\newtheorem{question}[theorem]{Question}

\newtheorem{subclaim}[sublemma]{Claim}

\newcommand{\QED}{\end{proof}}

\def\proclaim[#1]{{\bf #1}}
\def\BF#1.{{\bf #1.}}
\newcommand{\url}[1]{{\tt #1}}

\newcommand{\Jonsson}{J\'{o}nsson}

\newcommand{\Godel}{G\"odel}

\newcommand{\Erdos}{Erd\H{o}s}
\newcommand{\Levy}{L\'{e}vy}
\renewcommand{\P}{{\mathbb P}}
\newcommand{\Q}{{\mathbb Q}}

\newcommand{\R}{{\mathbb R}}

\newcommand{\of}{\subseteq}
\newcommand{\ofnoteq}{\subsetneq}

\newcommand{\set}[1]{\{\,{#1}\,\}}

\newcommand{\dom}{\mathop{\rm dom}}

\newcommand{\ran}{\mathop{\rm ran}}

\newcommand{\cof}{\mathop{\rm cof}}
\newcommand{\Cof}{\mathop{\rm Cof}}

\newcommand{\Coll}{\mathop{\rm Coll}}

\newcommand{\Con}{\mathop{{\rm Con}}}
\newcommand{\image}{\mathbin{\hbox{\tt\char'42}}}
\newcommand{\plus}{{+}}
\newcommand{\plusplus}{{{+}{+}}}

\newcommand{\restrict}{\upharpoonright} 
\newcommand{\satisfies}{\models}

\newcommand{\gHOD}{\ensuremath{\mathord{{\rm g}\HOD}}} 

\newcommand{\Union}{\bigcup}

\newcommand{\intersect}{\cap}
\newcommand{\Intersect}{\bigcap}

\newcommand{\smalllt}{\mathrel{\mathchoice{\raise2pt\hbox{$\scriptstyle<$}}{\raise1pt\hbox{$\scriptstyle<$}}{\raise0pt\hbox{$\scriptscriptstyle<$}}{\scriptscriptstyle<}}}
\newcommand{\smallleq}{\mathrel{\mathchoice{\raise2pt\hbox{$\scriptstyle\leq$}}{\raise1pt\hbox{$\scriptstyle\leq$}}{\raise1pt\hbox{$\scriptscriptstyle\leq$}}{\scriptscriptstyle\leq}}}
\newcommand{\lt}{\smalllt}

\newcommand{\boolval}[1]{\mathopen{\lbrack\!\lbrack}\,#1\,\mathclose{\rbrack\!\rbrack}}
\def\[#1]{\boolval{#1}}

\newcommand{\UnderTilde}[1]{{\setbox1=\hbox{$#1$}\baselineskip=0pt\vtop{\hbox{$#1$}\hbox to\wd1{\hfil$\sim$\hfil}}}{}}
\newcommand{\Undertilde}[1]{{\setbox1=\hbox{$#1$}\baselineskip=0pt\vtop{\hbox{$#1$}\hbox to\wd1{\hfil$\scriptstyle\sim$\hfil}}}{}}
\newcommand{\undertilde}[1]{{\setbox1=\hbox{$#1$}\baselineskip=0pt\vtop{\hbox{$#1$}\hbox to\wd1{\hfil$\scriptscriptstyle\sim$\hfil}}}{}}
\newcommand{\UnderdTilde}[1]{{\setbox1=\hbox{$#1$}\baselineskip=0pt\vtop{\hbox{$#1$}\hbox to\wd1{\hfil$\approx$\hfil}}}{}}
\newcommand{\Underdtilde}[1]{{\setbox1=\hbox{$#1$}\baselineskip=0pt\vtop{\hbox{$#1$}\hbox to\wd1{\hfil\scriptsize$\approx$\hfil}}}{}}

\newcommand{\st}{\mid}
\renewcommand{\th}{{\hbox{\scriptsize th}}}

\def\<#1>{\langle#1\rangle}

\newcommand{\cp}{\mathop{\rm cp}}

\newcommand{\ORD}{\mathop{{\rm ORD}}}

\newcommand{\ZFC}{{\rm ZFC}}
\newcommand{\ZF}{{\rm ZF}}

\newcommand{\KM}{{\rm KM}}

\newcommand{\NGBC}{{\rm NGBC}}
\newcommand{\NGB}{{\rm NGB}}

\newcommand{\AC}{{\rm AC}}

\newcommand{\HOD}{{\rm HOD}}

\newcommand{\cell}[1]{\boxit{\hbox to 17pt{\strut\hfil$#1$\hfil}}}
\newcommand{\head}[2]{\lower2pt\vbox{\hbox{\strut\footnotesize\it\hskip3pt#2}\boxit{\cell#1}}}
\newcommand{\boxit}[1]{\setbox4=\hbox{\kern2pt#1\kern2pt}\hbox{\vrule\vbox{\hrule\kern2pt\box4\kern2pt\hrule}\vrule}}
\newcommand{\Col}[3]{\hbox{\vbox{\baselineskip=0pt\parskip=0pt\cell#1\cell#2\cell#3}}}
\newcommand{\tapenames}{\raise 5pt\vbox to .7in{\hbox to .8in{\it\hfill input: \strut}\vfill\hbox to
.8in{\it\hfill scratch: \strut}\vfill\hbox to .8in{\it\hfill output: \strut}}}
\newcommand{\Head}[4]{\lower2pt\vbox{\hbox to25pt{\strut\footnotesize\it\hfill#4\hfill}\boxit{\Col#1#2#3}}}
\newcommand{\Dots}{\raise 5pt\vbox to .7in{\hbox{\ $\cdots$\strut}\vfill\hbox{\ $\cdots$\strut}\vfill\hbox{\
$\cdots$\strut}}}
\newcommand{\df}{\it} 
\DeclareMathOperator{\rk}{rank}
\DeclareMathOperator{\tcl}{tcl}
\begin{document}
\author[Hamkins]{Joel David Hamkins}
\address{J. D. Hamkins, Mathematics,
The Graduate Center of The City University of New York, 365
Fifth Avenue, New York, NY 10016 \& Mathematics, The
College of Staten Island of CUNY, Staten Island, NY 10314}
\email{jhamkins@gc.cuny.edu, http://jdh.hamkins.org}
\thanks{The research of the first author has been
supported in part by grants from the CUNY Research
Foundation, the Simons Foundation and the National Science Foundation. Commentary concerning this paper can be made at http://jdh.hamkins.org/generalizationsofkuneninconsistency.}
\author[Kirmayer]{Greg Kirmayer}
\address{Greg Kirmayer}
\email{kirmayerg@member.ams.org}
\author[Perlmutter]{Norman Lewis Perlmutter}
\address{Norman Perlmutter, Mathematics,
The Graduate Center of The City University of New York, 365
Fifth Avenue, New York, NY 10016}
\email{Norman314@gmail.com}

\begin{abstract}
We present several generalizations of the well-known Kunen
inconsistency that there is no nontrivial elementary
embedding from the set-theoretic universe $V$ to itself.
For example, there is no elementary embedding from the
universe $V$ to a set-forcing extension $V[G]$, or
conversely from $V[G]$ to $V$, or more generally from one set-forcing
ground model of the universe to another, or between any two
models that are eventually stationary correct, or from $V$
to $\HOD$, or conversely from $\HOD$ to $V$, or indeed from
any definable class to $V$, among many other possibilities
we consider, including generic embeddings, definable
embeddings and results not requiring the axiom of choice.
We have aimed in this article for a unified presentation
that weaves together some previously known unpublished or
folklore results, several due to Woodin and others, along
with our new contributions.
\end{abstract}

\title{Generalizations of the Kunen Inconsistency}
\maketitle

The Kunen inconsistency
\cite{Kunen1971:ElementaryEmbeddingsAndInfinitaryCombinatorics},
the theorem showing that there can be no nontrivial
elementary embedding from the universe to itself, remains a
focal point of large cardinal set theory, marking a hard
upper bound at the summit of the main ascent of the large
cardinal hierarchy, the first outright refutation of a
large cardinal axiom. On this main ascent, large cardinal
axioms assert the existence of elementary embeddings
$j:V\to M$ where $M$ exhibits increasing affinity with $V$
as one climbs the hierarchy. The $\theta$-strong cardinals,
for example, have $V_\theta\of M$; the
$\lambda$-supercompact cardinals have $M^\lambda\of M$; and
the huge cardinals have $M^{j(\kappa)}\of M$. The natural
limit of this trend, first suggested by Reinhardt, is a
nontrivial elementary embedding $j:V\to V$, the critical
point of which is accordingly known as a {\df Reinhardt}
cardinal. Shortly after this idea was introduced, however,
Kunen famously proved using the axiom of choice that there are no such embeddings
and hence no Reinhardt cardinals.

\begin{theorem}[The Kunen Inconsistency]\label{Theorem.NojVtoV}
There is no nontrivial elementary embedding $j:V\to V$.
\end{theorem}

In this article, we present several generalizations of this theorem,
thereby continuing what has been a small industry of
generalizations of this central result, including Harada
\cite[p. 320-321]{Kanamori2004:TheHigherInfinite2ed},
Woodin \cite[p. 322]{Kanamori2004:TheHigherInfinite2ed},
Zapletal \cite{Zapletal1996:ANewProofOfKunenInconsistency}
and Suzuki \cite{Suzuki1998:NojVtoVinV[G],
Suzuki1999:NoDefinablejVtoVinZF}. In order to emphasize a
coherent theme of mathematical ideas, we give a unified
presentation that includes some previously known
generalizations and unpublished folklore results along with
our new contributions. A fitting alternative title for our
paper might therefore have been ``Generalizations of
generalizations of the Kunen inconsistency.'' Among the
generalizations of the Kunen inconsistency we establish in
this article, several due to Woodin and others, are the
following:
\begin{enumerate}
 \item There is no nontrivial elementary embedding
     $j:V[G]\to V$ of a set-forcing extension of the
     universe to the universe, and neither is there
     $j:V\to V[G]$ in the converse direction.
 \item More generally, there is no nontrivial
     elementary embedding between two set-forcing ground models of
     the universe.
 \item More generally still, there is no nontrivial
     elementary embedding $j:M\to N$ when both $M$ and
     $N$ are eventually stationary correct.
 \item There is no nontrivial elementary embedding
     $j:V\to \HOD$, and neither is there $j:V\to M$ for
     a variety of other definable classes, including
     $\gHOD$ and the $\HOD^\eta$, $\gHOD^\eta$.
 \item If $j:V\to M$ is elementary, then $V=\HOD(M)$.
 \item There is no nontrivial elementary embedding
     $j:\HOD\to V$.
 \item More generally, for any definable class $M$,
     there is no nontrivial elementary embedding
     $j:M\to V$.
 \item There is no nontrivial elementary embedding
     $j:\HOD\to\HOD$ that is definable in $V$ from
     parameters.
\end{enumerate}
This list is just a selection; all the details and
additional more refined generalizations appear in the
subsequent theorems of this article, including other
natural definable classes, such as the iterated
$\HOD^\eta$, the generic-$\HOD$ and $\gHOD^\eta$, generic
embeddings, definable embeddings and results not requiring
the axiom of choice.

%

\section{A few metamathematical preliminaries} \label{Section.Preliminaries}

Before getting to the actual generalizations of the Kunen
inconsistency, let us begin by dispelling a few
metamathematical clouds that occasionally obscure the large
cardinal summit of the Kunen inconsistency. We should like
briefly to clarify these metamathematical issues.

The first concerns the fact that the Kunen inconsistency is
explicitly a second-order claim, as the purported embedding
$j$ that it rules out would clearly be a proper class of
some kind. In particular, the statement of the Kunen
inconsistency does not seem directly to be expressible in
the usual first-order language of set theory, as it
quantifies over second-order objects: {\it ``There is no
$j$ such that\ldots.''} So how and in which theory shall we
take it as a precise mathematical claim?

To be sure, many large
cardinal notions are characterized by second-order
assertions that turn out to have first-order equivalent
formulations, which can be treated in \ZFC. For example, a
cardinal $\kappa$ is measurable if it is the critical point
of an elementary embedding of $V$ into a transitive class,
and this is equivalent to the first-order assertion that
there is a nonprincipal $\kappa$-complete ultrafilter on
$\kappa$. But it is easy to see that there can be no
corresponding consistent first-order formulation of
Reinhardt cardinals, since if $\kappa$ is the least
Reinhardt cardinal, then by elementarity $j(\kappa)$ will
also be least, an immediate contradiction. More generally,
there can be no consistent first-order property
$\varphi(\kappa)$ that implies that $\kappa$ is Reinhardt,
because if $\kappa$ is the least cardinal with that
property, then by elementarity so is $j(\kappa)$, again a
contradiction. So there is no completely first-order
account of the Reinhardt cardinal concept, and the issue is
how then we are to formalize Reinhardt cardinals and the
Kunen inconsistency statement in some second-order manner.
Although there are a variety of satisfactory resolutions of
this issue, aligning with the various treatments of classes
and proper classes that are available in set theory, it
turns out that they are not equally efficacious, for some
provide a greater substance for the Kunen inconsistency
than others.

One traditional approach to classes in set theory, working
purely in \ZFC, is to understand all talk of classes as a
substitute for the first-order definitions that might
define them. In this formulation, the Kunen inconsistency
becomes a theorem scheme, asserting that no formula defines
(with parameters) a nontrivial elementary embedding of the
universe to itself. Thus, for each first-order formula
$\psi$ in the language of set theory, we have the theorem
that for no parameter $z$ does the relation $\psi(x,y,z)$
define a function $y=j(x)$ that is an elementary embedding
from $V$ to $V$. This is the approach used in
\cite{Kanamori1997:TheHigherInfinite,
Kanamori2004:TheHigherInfinite2ed}.\footnote{``As the
quantification $\forall j$ over classes $j$ cannot be formalized in ZFC,
this result can only be regarded as a schema of theorems,
one for each $j$''  \cite[p.
319]{Kanamori2004:TheHigherInfinite2ed}, see also \cite[p. 319]
{Kanamori1997:TheHigherInfinite}. The author also
notes, however, that the purely first-order strengthening
of the theorem to the assertion that there is no nontrivial
$j:V_{\lambda+2}\to V_{\lambda+2}$ is expressible and
provable purely in \ZFC.}

Our view is that this way of understanding the Kunen
inconsistency does not convey the full power of the
theorem. Part of our reason for this view is that if one is
concerned only with such definable embeddings $j$ in the
Kunen inconsistency, then in fact there is a far easier
proof of the result, simpler than any of the traditional
proofs of it and making no appeal to any infinite
combinatorics or indeed even to the axiom of choice. We
explain this argument in theorem
\ref{Theorem.NoDefinablej:VtoVinZF}.

Instead, a fuller power for the Kunen inconsistency seems
to be revealed when it is understood as a claim in a true
second-order set theory, such as von
Neumann-G\"odel-Bernays set theory \NGBC\ (see
 \cite[p. 70]{Jech:SetTheory3rdEdition}; this theory is also commonly known as
G\"odel-Bernays set theory) or Kelley-Morse \KM\ set
theory. Kunen himself understood his result to be
formalized in \KM, writing:
\begin{quote}
      It is intended that our results be formalized within the second~order
      Morse-Kelley set theory (as in the appendix to Kelley \cite{Kelley1965:GeneralTopology}),
      so that statements involving the satisfaction predicate for class models
      can be expressed.
      (\cite[p. 407]{Kunen1971:ElementaryEmbeddingsAndInfinitaryCombinatorics})
\end{quote}
Meanwhile, as we shall explain below, it turns out that for
the purposes of the Kunen inconsistency, one can
sufficiently express the satisfaction relation and prove
the theorem in the strictly weaker theory \NGBC, even in
the subtheory $\NGB+\AC$, or in the fragment of this we
denote by $\ZFC(j)$ below.

In these second-order theories, one distinguishes between
the first-order objects, the sets, and the second-order
objects, the classes, (although there are elegant
economical accounts that unify the treatment purely in
terms of classes). The crucial difference between the
theories is that \NGB\ includes the replacement and
separation axioms only for formulas having only first-order
quantifiers, that is, quantifying only over sets, allowing
finitely many class parameters, whereas \KM\ allows
formulas into the schemes that quantify also over classes.
The theories \NGBC\ and \KM\ include also a global choice
principle, whereas the theory $\NGB$ omits any choice
principle.

The \NGBC\ theory is conservative over \ZFC, meaning that
any first-order assertion about sets provable in \NGBC\ is
also provable in \ZFC, and this can be easily proved by
expanding any model of \ZFC\ to a model of \NGBC\ by adding
a generic global choice class, if necessary, and then
interpreting the second-order part to consist precisely of
the classes that are definable from this class and set
parameters. In particular, \NGBC\ and \ZFC\ are
equiconsistent. The theory \KM, in contrast, is strictly
stronger than \ZFC\ in consistency strength (if
consistent), because it proves that there is a satisfaction
predicate for first-order truth, and indeed, that there is
a satisfaction predicate for first-order truth relativized
to any class parameter. So \KM\ proves $\Con(\ZFC)$ and
$\Con(\ZFC+\Con(\ZFC))$ and so on transfinitely, and
therefore is not conservative over \ZFC, if this theory is
consistent.

In some of our arguments below, we will use the forcing
method in the \NGBC\ context, and so let us remark without
elaboration that the usual theory of forcing goes through
fine in this second-order setting: the classes of the
forcing extension are obtained from the classes of the
ground model by interpreting them in the usual name
fashion, \NGBC\ is preserved and all the usual forcing
technology works as expected. We say that $M$ is a ground model of $N$ if the latter can be realized as a forcing extension $N=M[G]$ for some $M$-generic filter $G\of\P\in M$. In order to emphasize that the forcing should be a set in $M$, we will sometimes say that $M$ is a set-forcing ground model.

Proving the Kunen inconsistency in \NGBC\ appears to give a
stronger result than either the \ZFC\ scheme approach or
the \KM\ approach, the former because it rules out not only
the definable embeddings, but also the possibility that a
non-definable embedding $j:V\to V$ may arise as a class in
an \NGBC\ model, and the latter simply because \NGBC\ is a
weaker theory than \KM\ and closer to \ZFC. The easy proof
of the definable embedding version of the Kunen
inconsistency, as in theorem
\ref{Theorem.NoDefinablej:VtoVinZF}, does not seem to
generalize to establish the stronger \NGBC\ result.
Therefore, for the rest of this article, unless otherwise
stated, we shall work formally in \NGBC\ set theory. But
actually, none of our arguments needs the global version of
choice, and so $\NGB+\AC$ suffices for us; meanwhile, when
we say below, ``Do not assume \AC,'' we intend to work in
\NGB. As the class $j$ and the ones definable from it will
be the only classes we need to consider, we could
alternatively formalize the presentation of our theorems in
the theory $\ZFC(j)$, the intermediate subtheory of $\NGB+\AC$
where one has fixed a single class predicate for $j$, which
is allowed to appear in the formulas of the replacement and
separation axiom schemes, or in $\ZF(j)$ when we do not use
\AC. When $M$ is a model of $\ZFC$, we shall say that $j$ is a
class of $M$ to mean that $(M, j)$ is a model of $\ZFC(j)$.

The second metamathematical issue concerning the Kunen
inconsistency that we would like to discuss, which arises
whether one uses the \ZFC\ theorem scheme approach, the
\NGBC\ approach or the \KM\ approach, is that the theorem
involves the hypothesis that $j$ is an elementary
embedding, and it is not immediately clear how to express
such a hypothesis in our language. Na\"ively, the assertion
that $j:V\to V$ is elementary is expressed most plainly by
a scheme of statements, those of the form $\forall x\,
[\varphi(\vec x)\longleftrightarrow \varphi(j(\vec x))]$,
rather than by a single statement. But a scheme does not
seem to serve the purpose here, because the assertion that
$j$ is elementary appears negatively in the
theorem---either as the antecedent of an implication or, in
the contrapositive, as the claim ultimately that $j$ is not
elementary---whereas the negation of a scheme is not
generally expressible even as a scheme. So we cannot seem
to use a scheme account of elementarity in order to find a
coherent statement of the theorem, even as a scheme, and
even when $j$ itself is assumed to be defined by a given
formula $\psi$. So again, how are we precisely to express
the theorem?

Kunen observed that this issue is addressed in \KM\ set
theory by the fact that \KM\ proves the existence of a
class satisfaction predicate for firstorder truth, by
means of which the elementarity of $j$ is expressible. Meanwhile, at
around the same time as the Kunen inconsistency, set-theorists realized how to express the elementarity of
$j$ in the weaker theory \NGB\ by making use the observation that every $\Delta_0$-elementary
cofinal embedding of models of \ZF\ is fully elementary. An
embedding $j:M\to N$ of transitive classes is {\df cofinal}
if for every $y\in N$ there is $x\in M$ with $y\in j(x)$.
Equivalently, $N=\Union j\image M$.

\begin{lemma}[Gaifman \cite{Gaifman1974:ElementaryEmbeddingsOfModelsOfSetTheoryAndCertainSubtheories}] \label{Lemma.Gaifman}
If $j:M\to N$ is $\Delta_0$-elementary and cofinal, where
$M$ satisfies \ZF, then $j$ is fully elementary.
\end{lemma}

The conclusion of the lemma, that $j$ is fully elementary,
is expressed as a scheme, consisting of the assertions
$\forall x\in M\,[\varphi^M(x)\iff\varphi^N(j(x))]$ for
every formula $\varphi$, and so the lemma is technically a
lemma scheme, asserting that any such embedding has any
desired degree of elementarity. One may regard the lemma
instead as the scheme asserting for each natural number $n$
that $j$ is $\Sigma_n$-elementary, which we may prove by
meta-theoretic induction on $n$. The atomic and Boolean
combination cases are easy, as is the forward direction of
the extensional case. The backward direction uses the
cofinality hypothesis to transform an unbounded existential
to a bounded existential, which reduces to the inductive hypothesis.
This final step of the argument is easy when one assumes that both $M$ and $N$ satisfy \ZF, since one can appeal to the absorption of bounded quantification by $\Sigma_n$ assertions in a \ZF-provably canonical manner. Gaifman's observation was that, in fact, it can be carried out even if one assumes only that $M$ satisfies \ZF, or even less. For example, \cite{GitmanHamkinsJohnstone:WhatIsTheTheoryZFC-Powerset?} provides full details for the context of $\ZFC^-$, set theory without the power set axiom. Also, \cite[p. 45]{Kanamori2004:TheHigherInfinite2ed} has some further details, including information about the fact that if $M$ and $N$ are transitive proper class models of \ZF\
containing all the ordinals and $j:M\to N$ is
$\Sigma_1$-elementary, then in fact the cofinality
assumption of lemma \ref{Lemma.Gaifman} follows for free.
This is because $j(V_\alpha^M)=V_{j(\alpha)}^N$, as the
relation ``$x=V_\alpha$'' has complexity $\Pi_1$, and since
$\alpha\leq j(\alpha)$, these sets union up to $N$, which
implies that $j$ is cofinal.

One important difference between the \NGBC\ approach to
expressing the elementarity of $j$ via Gaifman's theorem
and the \KM\ approach using a satisfaction predicate is
that in the former account, one derives any desired
instance of elementarity by meta-theoretic induction from
the assumption that $j$ is $\Delta_0$ elementary and
cofinal, whereas in \KM\ one uses the internal assertion
that $j$ is elementary with respect to the first-order
satisfaction class that \KM\ proves to exist. In
particular, in the \KM\ context, such embeddings are also
elementary with respect to the possibly nonstandard
formulas, and one can imagine inductive arguments that rely
internally on $\Sigma_n$-elementarity for every natural
number $n$, a possibility that it seems could not be
carried out in \NGBC\ approach via Gaifman's theorem,
because in that approach one has such full elementarity
only as a meta-theoretic scheme.

Note that if $M$ is an $\NGBC$ model and $j:M\to N$ is
elementary in the sense of lemma \ref{Lemma.Gaifman}, then
we may extend the domain of $j$ to include every \NGBC\
class $A$ (including in particular $A=j$, if this should be a class of $M$) by defining
$j(A)=\Union_{\alpha\in\ORD}j(A\intersect V_\alpha)$. This
extended embedding remains cofinal and
$\Delta_0$-elementary from the structure $\< M,{\in},A>\to \<N,\in,j(A)>$, and so by
the argument of Gaifman's lemma, it is fully elementary in
the expanded context, meaning that for any set $x$ and
class $A$ we have $M\satisfies \varphi(x, A) \iff N
\satisfies \varphi(j(x), j(A))$ for any formula $\varphi$
with only first-order quantifiers. Since we have used $j(A)$ on the right, this is slightly weaker from saying that $j$ is elementary in the language with $A$, if one should take this phrase to mean the stronger property that $A$ is interpreted the same in both $M$ and $N$. In particular, we rarely expect an embedding $j$ to be elementary in the language with $j$ itself, if $j$ is interpreted as $j$ in both the domain and the range, since the critical point $\kappa=\cp(j)$ is definable from $j$, but it is not in the range of $j$, which would contradict this stronger form of elementarity with respect to $j$. Rather, if $j:M\to N$ and $j$ is a class of $M$, then what we have is elementarity $j:\<M,\in,j>\to \<N,\in,j(j)>$, using $j(j)$ on the right.

A third metamathematical issue, which does not arise for
the Kunen inconsistency itself, but which does arise for
several of the generalizations of it, is the question of
how to express that a given transitive proper class $M$ is
in fact a model of \ZF\ or of \ZFC. After all, a na\"ive
formalization of this would seem to involve a scheme of
assertions, asserting that $M\satisfies\psi$ for every
axiom $\psi$. Nevertheless, it is well-known that for a
transitive proper class one can reduce this entire scheme
to a single first-order assertion about $M$ as described in lemma \ref{Lemma.ZFFirstOrderExpressible} below. A
transitive class $M$ is {\df almost universal} if for every set $x$ such that $x\of M$, there exists $y \in M$ such that $x\of y$.

\begin{lemma}[{\cite[Theorem 13.9]{Jech:SetTheory3rdEdition}}]\label{Lemma.ZFFirstOrderExpressible}
A transitive proper class $M$ is a model of \ZF\ if and
only if it is closed under the finitely many \Godel\
operations and is almost universal. This property of $M$ is
expressible in a single first-order assertion using class
parameter $M$.
\end{lemma}

In light of all these observations, let us adopt the
following conventions as a matter of definition. For a
transitive proper class $M$, by the phrase ``$M$ satisfies
ZF,'' we mean that $M$ satisfies the properties mentioned
in lemma \ref{Lemma.ZFFirstOrderExpressible}; by ``$M$
satisfies ZFC,'' we mean that $M$ satisfies these and also
the axiom of choice. When $M$ and $N$ are transitive proper
class models of \ZF, then by the phrase ``$j:M\to N$ is an
elementary embedding,'' we mean that $j$ is a
$\Delta_0$-elementary cofinal embedding, which can be
expressed by a single \NGBC\ assertion. The embedding $j$
is {\df trivial} if $M=N$ and $j$ is the identity map, and
otherwise it is {\df nontrivial}. By these conventions, the
Kunen inconsistency is expressed by a single assertion of
\NGBC\ set theory. By similar means, it can be equivalently
formulated in $\ZFC(j)$ set theory as the assertion that
$j$ is not a nontrivial elementary embedding from $V$ to
$V$.

We conclude this section by recalling the existence of
critical points for the embeddings we shall subsequently
consider.

\begin{lemma}\label{Lemma.CriticalPoint}
Suppose that $j: M \to N$ is a nontrivial elementary
embedding of transitive class models $M$ and $N$ satisfying
\ZF\ and having the same ordinals. If either
\begin{enumerate}
 \item $M\satisfies\AC$, or
 \item $N\of M$, or
 \item $M\of N$ and $M$ is definable in $N$ without
     parameters or with parameters in the range of $j$,
\end{enumerate}
then there
is a least ordinal $\kappa$ such that
$\kappa<j(\kappa)$. Furthermore, $\kappa$ is a regular
uncountable cardinal in $M$.
\end{lemma}

\begin{proof}
Statements (1) and (2) are handled in the standard set
theory texts, such as \cite[p.
45]{Kanamori2004:TheHigherInfinite2ed}, where we refer the
reader for details. For statement (3), note first that this
claim is technically a scheme, making the assertion
separately of every possible definition of $M$ in $N$. To
see that it is true, suppose that $M$ is a transitive class
defined in $N$ by $x\in M\iff\varphi(x, j(b))$, and that
$j:M\to N$ is elementary and
$j(\alpha)=\alpha$ for every ordinal $\alpha$. We shall
argue that $j$ is trivial, meaning $M=N$ and $j$ is the
identity. For this, we prove that $u\in M$ and $j(u)=u$ by
induction on the $\in$-rank of $u\in N$. Suppose that this
is true for all elements of $N$ of lower rank than a set
$u\in N$, having rank $\alpha$. If $u\in M$, then by our
induction assumption we have $j(v)=v$ for all $v\in u$, and
so $u\of j(u)$. Conversely, if $w\in j(u)$, then since
$j(\alpha)=\alpha$, it follows that $w$ has rank less than
$\alpha$, and so by induction we conclude that $w\in M$ and
$j(w)=w$, from which it follows that $w\in u$. So $j(u)=u$,
as desired. It remains to see that every $u$ in $N$ of rank
$\alpha$ is in $M$. If not, then $N$ thinks there is some
$w$ of rank $\alpha$ with $\neg\varphi(w, j(b))$. Since $j$
fixes $\alpha$, this implies by elementarity that
$M$ thinks there is $u$ of rank $\alpha$ with
$\neg\varphi(u,b)$. We have already proved that $j(u)=u$
for all such $u$ in $M$. Thus, by elementarity again, $N$
satisfies $\neg\varphi(u, j(b))$, contrary to the fact that
$u\in M$. So we have proved that $M=N$ and $j$ is the
identity, as desired.
\end{proof}

We note that it is not universally true that nontrivial
embeddings of transitive \ZF\ models must have critical
points. Indeed, it is perfectly possible to have a nontrivial elementary embedding $j:M\to N$ of transitive class models of \ZF\ with no critical point.
Andr\'es Caicedo \cite{Caicedo2003:Dissertation}, for example, proves that
if $V[G]$ is the forcing extension obtained by adding $\omega_1$ many Cohen reals, and $G_0$ is the filter obtained from an uncountable subcollection of them, then there is an elementary embedding $j:L(\R^{V[G_0]})\to L(\R^{V[G]})$ which is the identity on the ordinals (and a complete argument appears on mathoverflow at \cite{MO81568:ElementaryEmbeddings}). Neeman and Zapletal
\cite{NeemanZapletal1998:ProperForcingAndAbsolutenessInL(R)},
\cite{NeemanZapletal2001:ProperForcingAndL(R)} showed that if there is a weakly compact Woodin
cardinal, then in every small proper forcing extension
$V[G]$, there is an elementary embedding $j:L(\R^V)\to
L(\R^{V[G]})$ that fixes every ordinal, and therefore has
no critical point, but if reals are added, then since
$j(\R^V)=\R^{V[G]}$, the embedding must be nontrivial.

\section{Elementary embeddings between V and V[G]}\label{Section.jbetweenVandV[G]}

Let us now begin to generalize the Kunen inconsistency by
considering the possibility that nontrivial elementary
embeddings $j$ might be added by forcing. For example,
perhaps one might think that in special circumstances,
there could be a forcing extension $V[G]$ in which we could
find a nontrivial elementary embedding $j:V[G]\to V$ or
conversely $j:V\to V[G]$.  The case of an embedding
$j:V[G]\to V$ is quite natural to consider, because from
the perspective of the forcing extension $V[G]$, this would
be an embedding of the universe into a certain transitive
class, a situation that na\"ively resembles the typical
large cardinal situation. And the question of whether there
can be embeddings $j:V\to V[G]$ has arisen independently
several times in the set-theoretic community. Woodin ruled
out embeddings of the form $j:V[G]\to V$, as we shall
presently explain in theorem \ref{Theorem.Noj:V[G]toV} and
corollary \ref{Corollary.VisNotForcingExtensionOfM}. A
generalization of his method, in theorem
\ref{Theorem.Noj:VtoV[G]}, rules out the converse sort of
embedding $j:V\to V[G]$. These results directly generalize
the Kunen inconsistency, which is simply the case of
trivial forcing $V[G]=V$. Furthermore, they are themselves
generalized by and special cases of theorem
\ref{Theorem.NojBetweenGrounds}, asserting that there is no
nontrivial elementary embedding between two set-forcing ground models
of the universe, a result that will itself be generalized
in subsequent sections.

Although logically we could skip to the most general
results, we find the historical progression informative and
also prefer to outline the basic methods in their easiest
cases, as the generalizations eventually become complex.
Nevertheless, we shall economize by giving only brief
accounts in the parts of later arguments that follow a
method we will have had explained earlier in detail.

In the following, let $\Cof_\delta$ be the class of
ordinals having cofinality $\delta$ and
$\Cof_\delta\gamma=\gamma\intersect\Cof_\delta$ be the set
of such ordinals below $\gamma$. We shall make extensive
use of the Ulam-Solovay stationary partition theorem, asserting that every stationary subset of a
regular uncountable cardinal $\kappa$ can be partitioned
into $\kappa$ many stationary subsets (see \cite[p.
95]{Jech:SetTheory3rdEdition}), a result proved for successor cardinals by Ulam and extended to all cardinals by Solovay. Note that a mild
generalization of the theorem shows that if $\omega <
\kappa = \cof(\lambda)$, then every stationary subset of
$\lambda$ can be partitioned into $\kappa$ many disjoint
stationary sets. To see this, pick a normal
$\kappa$-sequence, $F$, in $\lambda$. Then $C\of \kappa$ is
a club in $\kappa$ if and only if $F\image C$ is a club in
$\lambda$, and $S$ is stationary in $\kappa$ if and only if
$F \image S$ is a stationary in $\lambda$. The
generalization follows easily.

\begin{theorem}[Woodin]\label{Theorem.Noj:V[G]toV}
In any set-forcing extension $V[G]$, there is no nontrivial
elementary embedding $j:V[G]\to V$.
\end{theorem}

\begin{proof} Suppose
towards contradiction that $V[G]$ is a set-forcing
extension of $V$, obtained by forcing with $\P\in V$ to add
the $V$-generic filter $G\of\P$, and that $j:V[G]\to V$ is
a nontrivial elementary embedding in $V[G]$, where $j$ is a
class in $V[G]$. By lemma \ref{Lemma.CriticalPoint} applied
in $V[G]$, it follows that $j$ has a critical point,
$\kappa$, which of course will be a measurable cardinal in
$V[G]$.

First, we find an ordinal $\lambda$ above both $\kappa$ and
$|\P|$ such that $j(\lambda)=\lambda$ and
$\cof(\lambda)^V=\cof(\lambda)^{V[G]}=\omega$. To find
$\lambda$, fix any ordinal $\delta_0$ above $\kappa$ and
$|\P|$ with $\cof(\delta_0) = \omega$ in $V$. If
$j(\delta_0)=\delta_0$, then we have found what we wanted.
Otherwise, $\delta_0<j(\delta_0)$ and we may recursively
define $\delta_{n+1}=j(\delta_n)$ and
$\lambda=\sup\set{\delta_n\st n\in\omega}$. Since $\lambda$
is the supremum of the increasing $\omega$-sequence
$\<\delta_n\st n<\omega>$, it follows that $j(\lambda)$ is
the supremum of $j(\<\delta_n\st
n<\omega>)=\<j(\delta_n)\st n<\omega>=\<\delta_{n+1}\st
n<\omega>$. Thus, $j(\lambda)=\lambda$ and
$\cof(\lambda)=\omega$ in both $V[G]$ and $V$, as desired.

Since $\lambda>|\P|$, it follows that $\P$ satisfies the
$\lambda$-c.c.~in $V$ and so $\lambda^\plus$ has the same
meaning in both $V$ and $V[G]$. Thus,
$j(\lambda^\plus)=\lambda^\plus$. The set
$\Cof_\omega\lambda^\plus$ of $V[G]$ is stationary in
$V[G]$, and so by the Ulam-Solovay Theorem, we may partition it
into $\kappa$ many disjoint stationary sets:
$\Cof_\omega\lambda^\plus=\bigsqcup_{\alpha<\kappa}
S_\alpha$ in $V[G]$. Let $\vec
S=\<S_\alpha\st\alpha<\kappa>$ and consider $S^*=j(\vec
S)(\kappa)$. This is a stationary subset of
$(\Cof_\omega\lambda^\plus)^V$ in $V$, and it is disjoint
from $j(S_\alpha)$ for all $\alpha<\kappa$. Let
$C=\set{\beta<\lambda^\plus\st j\image\beta\of\beta}$,
which is a club subset of $\lambda^\plus$ in $V[G]$. Since
$\lambda^\plus$ is above the size of the forcing, it
follows that there is a club subset $D\of C$ with $D\in V$.
Thus, there is some $\beta\in D\intersect S^*$. Since
$S^*\of\Cof_\omega\lambda^\plus$, it follows that
$\cof(\beta)=\omega$ in $V$ and hence also in $V[G]$. Since
$j\image\beta\of\beta$ and $\cof(\beta)=\omega$, it follows
easily that $j(\beta)=\beta$. But since
$\beta\in(\Cof_\omega\lambda^\plus)^{V[G]}$, there must be
some $\alpha<\kappa$ with $\beta\in S_\alpha$, since these
sets form a partition. Thus, $\beta=j(\beta)\in
j(S_\alpha)=j(\vec S)(j(\alpha))$ and $\beta\in S^*=j(\vec
S)(\kappa)$, contrary to the fact that the sets appearing
in $j(\vec S)$ are disjoint. So there can be no such
embedding $j$, and the proof is complete.\end{proof}

An equivalent formulation of theorem
\ref{Theorem.Noj:V[G]toV}, stated from the point of view of
the extension, is the following.

\begin{corollary}\label{Corollary.VisNotForcingExtensionOfM}
If $j:V\to M$ is a nontrivial elementary embedding in $V$,
then $V$ is not a set-forcing extension of $M$.
\end{corollary}
\noindent In other words, if $j:V\to M$, then $M$ is not a set-forcing
ground model of $V$.

Let us turn now to the converse sort of embedding, from $V$
to $V[G]$, which we shall rule out by a generalization of
the method. In the case where $V[G]$ is the target of the
embedding rather than its domain, it no longer suffices to
consider $\Cof_{\omega} {\lambda^+}$ in the domain of the
embedding, because an ordinal of cofinality $\omega$ in
$V[G]$ might have a higher cofinality in $V$, and the
argument breaks down when applied directly. We shall repair
this issue by considering a stationary set of ordinals
having much larger cofinality. Another subtle point is that
the previous argument to obtain the fixed point $\lambda$
no longer succeeds exactly as before in the new context,
because the sequence $\langle \delta_n \st
n\in\omega\rangle$ may not be an element of $V$; but again
a repair will be provided by considering a longer sequence.

Attribution for this next theorem is not clear to us. Woodin reportedly proved it along with theorem \ref{Theorem.Noj:V[G]toV} while he was a graduate student in the early 1980s. But also, Matt Foreman mentioned to the first author that he discussed a version of the theorem with
Mack Stanley and Sy Friedman around the same time, but their proof
was evidently different than ours, and unfortunately the result was not published.\footnote{Part of their focus was reportedly on
the extent to which the result generalized to class
forcing. For example, they considered the case of class
forcing extensions by amenable class forcing. Foreman
mentioned that Woodin has an example of forcing using a
class version of non-stationary tower forcing where $j:V\to
V[G]$, but $V[G]$ does not have \ZFC\ for the predicate
$V$, a result cited in \cite[p.1594]{Suzuki1999:NoDefinablejVtoVinZF} and in \cite[p. 1091]{VickersWelch2001:jMtoV}.} Suzuki proved a theorem implying our theorem \ref{Theorem.Noj:VtoV[G]} in \cite[p. 344]{Suzuki1998:NojVtoVinV[G]}, using a technique essentially the same as ours.\footnote{Suzuki proved that if there is $j: V \to M$ in $V[G]$, then $V \nsubseteq M$. This result is stronger than our theorem \ref{Theorem.Noj:VtoV[G]}, but weaker than our theorem \ref{Theorem.NojMtoNifStationaryCorrect}. Although Suzuki states in his introduction that his proof only concerns definable $j$, in fact his proof never uses that fact and can be formalized in NGBC.}
The question
about such embeddings has arisen several times since then,
however, and so we are pleased to provide a proof here.

\begin{theorem}\label{Theorem.Noj:VtoV[G]}
In any set-forcing extension $V[G]$, there is no nontrivial
elementary embedding $j:V\to V[G]$.
\end{theorem}

\begin{proof}Suppose towards contradiction that $V[G]$ is
a set-forcing extension, obtained by forcing with $\P$ to
add a $V$-generic filter $G\of\P$, and that $j:V\to V[G]$
is an elementary embedding and a class in $V[G]$. Since
$V\satisfies\ZFC$, it follows by lemma
\ref{Lemma.CriticalPoint} that $j$ has a critical point
$\kappa$, which it is easy to see must be a regular
uncountable cardinal in $V$.

We claim as in theorem \ref{Theorem.Noj:V[G]toV} that there
is an ordinal $\lambda$ above $\kappa$ and $|\P|$ with
$\cof(\lambda)^V=\omega$ and $j(\lambda)=\lambda$. To see
this, consider the class function
$j\restrict\ORD:\ORD\to\ORD$ in $V[G]$. As before, start at
any $\delta_0$ greater than both $\kappa$ and $|\P|$ and
with $\cof(\delta_0)^V=\omega$. If $j(\delta_0)=\delta_0$,
we are done. Otherwise, define
$\delta_{\alpha+1}=j(\delta_\alpha+1)$ and
$\delta_\xi=\sup_{\alpha<\xi}\delta_\alpha$ for limit
$\xi$. Because we added $1$ at each step, this is a
strictly increasing, continuous sequence of ordinals, and
by construction, $j\image\delta_\xi\of\delta_\xi$ for any
limit $\xi$. Let $\gamma=|\P|^\plus$ and
$C=\set{\delta_\xi\st \xi<\gamma}$. This is a club subset
of $\delta_\gamma$, which has cofinality $\gamma$. Since
$\P$ satisfies the $\gamma$-c.c., it follows that there is
a club $D\of C$ with $D\in V$. Let $\lambda$ be the
$\omega^\th$ element of $D$. Thus, $\cof(\lambda)^V=\omega$
and since $\lambda\in C$, we know that
$j\image\lambda\of\lambda$. Since $\lambda$ is the supremum
of an $\omega$-sequence in $V$, it follows that
$j(\lambda)$ is the supremum of $j$ of that sequence, but
since $j\image\lambda\of\lambda$, this supremum is just
$\lambda$. So $j(\lambda)=\lambda$, as desired.

Since $j(\lambda)=\lambda$ and all cardinals above
$\lambda$ are preserved by the forcing, it follows that
$j(\lambda^\plus)=\lambda^\plus$ and
$j(\lambda^\plusplus)=\lambda^\plusplus$. Since
$\lambda^\plus$ is a regular cardinal larger than $|\P|$,
it follows that $\Cof_{\lambda^\plus}$ is absolute between
$V$ and $V[G]$. Since
$\Cof_{\lambda^\plus}\lambda^\plusplus$ is stationary, it
follows by the Ulam-Solovay theorem in $V$ that we may partition
it as a disjoint union:
$\Cof_{\lambda^\plus}\lambda^\plusplus=\bigsqcup_{\alpha<\kappa}
S_\alpha$ of $\kappa$ many stationary subsets $S_\alpha$.
Let $\vec S=\<S_\alpha\st\alpha<\kappa>$ and consider
$S^*=j(\vec S)(\kappa)$. By elementarity, this is a
stationary subset of
$\Cof_{\lambda^\plus}\lambda^\plusplus$ in $V[G]$, and
disjoint from $j(S_\alpha)$ for all $\alpha<\kappa$. Let
$C=\set{\beta<\lambda^\plusplus\st j\image\beta\of\beta}$,
which is a club subset of $\lambda^\plusplus$ in $V[G]$.
Thus, there is some $\beta\in S^*\intersect C$. In
particular, $\beta\in S^*\of\Cof_{\lambda^\plus}$ and so
$\cof(\beta)=\lambda^\plus$ in $V[G]$ and in $V$. If
$\beta$ is the supremum of $s\of\beta$, where $s$ has order
type $\lambda^\plus$, it follows that $j(\beta)$ is the
supremum of $j(s)$, having order type
$j(\lambda^\plus)=\lambda^\plus$. It follows that $j\image
s$ is unbounded in $j(s)$, and therefore since
$j\image\beta\of\beta$, it follows that the supremum of
$j(s)$ is $\beta$. Thus, $j(\beta)=\beta$. Since
$\beta\in\Cof_{\lambda^\plus}\lambda^\plusplus$, it must be
that $\beta\in S_\alpha$ for some $\alpha<\kappa$. So we
have a contradiction, just as in the proof of theorem
\ref{Theorem.Noj:V[G]toV}, since $\beta$ is in both $j(\vec
S)(\alpha)$ and in $j(\vec S)(\kappa)$, even though these
are disjoint.
\end{proof}

Stated from the perspective of the forcing extension, what
theorem \ref{Theorem.Noj:VtoV[G]} asserts is that if
$j:M\to V$ is a nontrivial elementary embedding from a
transitive class $M$ into the universe $V$, then $V$ is not
a set-forcing extension of $M$. Both theorem
\ref{Theorem.Noj:V[G]toV} and \ref{Theorem.Noj:VtoV[G]} are
therefore special cases of the following theorem.

\begin{theorem}\label{Theorem.NojBetweenGrounds}
If $M$ and $N$ are set-forcing ground models of $V$, then
there is no nontrivial elementary embedding $j : M \to N$.
\end{theorem}

In other words, if $M$ and $N$ have a common set-forcing
extension $M[G]=N[H]$, then in no such extension is there a
nontrivial elementary embedding $j:M\to N$. This theorem is
an immediate corollary of theorem
\ref{Theorem.NojMtoNifStationaryCorrect} in the next
section, since any ground model is stationary correct in
its forcing extension at regular cardinals above the size
of the forcing. An important special case of theorem
\ref{Theorem.NojBetweenGrounds}, obtained simply by
applying it in a forcing extension $V[G]$, is the generic
embedding version of the Kunen inconsistency.

\begin{corollary} \label{Corollary.NoGenericjVtoV}
In no set-forcing extension $V[G]$ is there a nontrivial
elementary embedding $j:V\to V$.
\end{corollary}

Indeed, theorem \ref{Theorem.NojMtoNifStationaryCorrect}
will show that there is no such class $j$ in any extension
of $V$ that is eventually stationary correct (this includes
many class forcing extensions), and the remarks after that
theorem generalize it still further.

\section{Elementary embeddings between eventually
stationary-correct
models}\label{Section.StationaryCorrectNew}

We define that a transitive class model $M$ of set theory
is {\df stationary correct} at a limit ordinal $\delta$ of
$M$ if every subset of $\delta$ that is stationary in $M$
remains stationary in the universe $V$. More generally, one model $M$ is stationary correct to another larger model $N$ at $\delta$, if every stationary subset of $\delta$ in $M$ is stationary in $N$. For example, after any forcing
of size less than a regular cardinal $\delta$, and indeed,
after any $\delta$-c.c.~forcing, every stationary subset of
$\delta$ in the ground model remains stationary in the
extension, because every club subset of $\delta$ in the
forcing extension will contain a club of the ground model
as a subset. After any set forcing, therefore, the ground
model is stationary correct in the extension at all regular
cardinals above the size of the forcing. Since in general
regular cardinals of $M$ can become singular in $V$, let us
adopt the convention that a subset $S\of\delta$ is
stationary in the case $\cof(\delta)=\omega$ if and only if
$S$ contains a final segment $(\alpha,\delta)$ for some
$\alpha<\delta$. In particular, in the extreme case that a
cardinal $\delta$ has uncounable cofinality in $M$ and
countable cofinality in $V$, then $M$ is not stationary
correct at $\delta$.

Before proving the theorem, we make a few elementary
observations about stationary correctness. First, a
transitive class $M$ is stationary correct at $\delta$ if
and only if it is stationary correct at $\cof(\delta)^M$.
The countable cofinality case is immediate; when the
cofinality is uncountable in $M$, then in either case of
stationary-correctness, it follows that $\delta$ also has
uncountable cofinality in $V$. We may fix a club subset
$C\of\delta$ in $M$ with order type $\cof(\delta)^M$ and
observe that a set $S\of\delta$ is stationary if and only
if $S\intersect C$ is stationary. Viewing $S\intersect C$
as a subset of $C$, this latter claim is equivalent to the
stationarity of a subset of $\cof(\delta)^M$, as desired.
Second, we note that if $M\of V$ is stationary correct at a
regular cardinal $\delta$ of $M$, then $\delta$ remains
regular in $V$. If not, let $\eta=\cof(\delta)^V<\delta$ be
the new smaller cofinality, which is a regular cardinal in
$V$, and let $C\of\delta$ be a cofinal $\eta$-sequence in
$V$. The case $\eta=\omega$ is easily ruled out by our
convention about stationary subsets of ordinals with
cofinality $\omega$, and so we may assume $\eta$ is
uncountable. The set $C'$ of limit points of $C$,
therefore, is club and consists entirely of ordinals of
cofinality less than $\eta$ in $V$. Since $\eta$ is regular
in $V$, no such ordinal can have cofinality $\eta$ in $M$,
and so $C'$ is disjoint from the set $(\Cof_\eta\delta)^M$.
This set, which is stationary in $M$, is thus no longer
stationary in $V$, contradicting our hypothesis that $M$
was stationary correct. Finally, third, the two previous
observations together imply that if $M\of V$ is stationary
correct at an ordinal $\delta$, then
$\cof(\delta)^M=\cof(\delta)^V$.

\begin{theorem}\label{Theorem.NojMtoNifStationaryCorrect}
Suppose that $M$ and $N$ are transitive class models of
\ZFC\ and both are stationary-correct at all sufficiently
large regular cardinals of $M$ and $N$. Then in $V$ there
is no nontrivial elementary embedding $j:M\to N$.
\end{theorem}

\begin{proof}Suppose that $j:M\to N$ is a nontrivial elementary
embedding and that $M$ and $N$ are transitive proper class
models of \ZFC, which are eventually stationary correct to
$V$. By lemma \ref{Lemma.CriticalPoint}, the embedding $j$
has a critical point $\kappa$. As in the previous proofs,
we shall begin by finding an ordinal $\lambda$ above the
critical point of $j$ with $j(\lambda)=\lambda$. To find
such an ordinal, suppose that $M$ and $N$ are stationary
correct at all regular cardinals of $M$ or $N$ above
$\theta$, and we may assume $\kappa<\theta$. The general
considerations before the theorem show that if a cardinal
should have cofinality above $\theta$ in any of the models
$M$, $N$ or $V$, then this cofinality is preserved to the
other models. Considering the operation of $j$ on the
ordinals, it is easy to see that $C=\set{\beta\in\ORD\st
j\image\beta\of\beta}$ is a closed unbounded proper class
of ordinals. Thus, there is some $\eta\in C$ with
$\theta<\cof(\eta)$. Since $(\Cof_\omega\eta)^M$ is
stationary in $M$, it remains stationary in $V$. Since
$C\intersect\eta$ is club in $\eta$, there is some
$\lambda\in C\intersect (\Cof_\omega\eta)^M$ with
$\theta\leq\lambda$. Thus, $j\image\lambda\of\lambda$ and
$\cof(\lambda)^M=\omega$. From this, it easily follows that
$j(\lambda)=\lambda$. Since cardinals above $\theta$ are
preserved, it follows that $\lambda^\plus$ and
$\lambda^\plusplus$ are the same in all three models, and
so $j(\lambda^\plus)=\lambda^\plus$ and
$j(\lambda^\plusplus)=\lambda^\plusplus$.

We continue the argument just as in theorem
\ref{Theorem.Noj:VtoV[G]}. Since cofinalities above
$\theta$ are computed correctly, the set
$\Cof_{\lambda^\plus}\lambda^\plusplus$ is absolute between
$M$, $N$ and $V$. By the Ulam-Solovay Theorem, there is a
partition
$\Cof_{\lambda^\plus}\lambda^\plusplus=\bigsqcup_{\alpha<\kappa}S_\alpha$
in $M$ into stationary sets $S_\alpha$. Let $\vec
S=\<S_\alpha\st\alpha<\kappa>$ and $S^*=j(\vec S)(\kappa)$.
By elementarity, $S^*$ is a stationary subset of
$\Cof_{\lambda^\plus}\lambda^\plusplus$ in $N$, and hence
also stationary in $V$. Since
$C\intersect\lambda^\plusplus$ is club in
$\lambda^\plusplus$, there is some $\beta\in C\intersect
S^*$. Thus, $j\image\beta\of\beta$ and
$\cof(\beta)=\lambda^\plus$ in $N$. Since this is above
$\theta$, it follows that $\cof(\beta)=\lambda^\plus$ also
in $V$ and $M$, and from this it follows as before that
$j(\beta)=\beta$. Since $\beta$ has cofinality
$\lambda^\plus$ in $M$, it follows that $\beta\in S_\alpha$
for some $\alpha<\kappa$, which implies that $\beta\in
j(S_\alpha)$ and in $S^*$, contradicting the fact that
these sets are disjoint.
\end{proof}

Because every ground model is stationary correct in its
set-forcing extensions above the size of the forcing, it
follows that theorem \ref{Theorem.NojBetweenGrounds} is a
special case of theorem
\ref{Theorem.NojMtoNifStationaryCorrect}, and therefore
theorem \ref{Theorem.NojMtoNifStationaryCorrect}
generalizes all the theorems of section
\ref{Section.jbetweenVandV[G]}. This theorem, however, also
allows us to rule out embeddings between certain class
forcing ground models, provided that they are eventually
stationary correct, and this includes many natural
class-forcing iterations.

It is clear that we may weaken the assumption that $M$ and
$N$ are eventually stationary correct in theorem
\ref{Theorem.NojMtoNifStationaryCorrect}, since once we
knew that the models agreed on $\lambda^\plus$ and
$\lambda^\plusplus$, then stationarity was used only to
ensure that the sets had a member in the particular club
class $C$, the class of closure points of $j$.

For example, for a class function $j$, we could define that
a class $M$ is {\df $j$-correct}, if it is eventually
correct about regular cardinals and there is a closed
unbounded class $D$ of ordinals such that if
$C=\set{\beta\in D\st j\image\beta\of\beta}$ is unbounded
in some $\gamma$, then every $S\of\gamma$ that $M$ thinks
is stationary contains an element of $C$. The concept of
stationary correctness in theorem
\ref{Theorem.NojMtoNifStationaryCorrect} and many of the
other theorems in this article could then be replaced with
the concept of $j$-correctness. For example, theorem
\ref{Theorem.NojMtoNifStationaryCorrect} would become the
claim that there is no nontrivial elementary embedding
$j:M\to N$, for which $M$ and $N$ are both $j$-correct.

\section{Embeddings from $V$ into $\HOD$ and
related models}\label{Section.intoHODetc}

Let us consider now the possibility of a nontrivial
elementary embedding $j:V\to\HOD$ of the universe into the
class of hereditarily ordinal-definable sets. After ruling
out such embeddings in theorem \ref{Theorem.Noj:VtoHOD}, we
shall generalize the result to several other related inner
models. We shall give a slightly modified version of
Woodin's original proof of this theorem, which will enable
us easily to extract additional information from it in the
various corollaries and theorems that we prove after it in
this section.

\begin{theorem}[Woodin]\label{Theorem.Noj:VtoHOD}
There is no nontrivial elementary embedding $j:V\to \HOD$.
\end{theorem}

\begin{proof}
We note briefly that we need not specifically assume \AC\
in this argument, as it follows trivially that $V$
satisfies \AC\ from the assumption that $j:V\to\HOD$ is
elementary, since \AC\ holds in $\HOD$. By lemma
\ref{Lemma.CriticalPoint}, the embedding $j$ has a critical
point $\kappa$, which must be a measurable cardinal in $V$.
Define the usual critical sequence by $\kappa_0=\kappa$ and
$\kappa_{n+1}=j(\kappa_n)$, and let
$\lambda=\sup_{n<\omega}\kappa_n$. As in the previous
proofs, it follows that $j(\lambda) = \lambda$ and also
that $j(\lambda^\plus)=\lambda^\plus$. By the Ulam-Solovay
Theorem, we may partition $\Cof_\omega\lambda^\plus$ into
$\lambda^\plus$ many disjoint stationary sets $\vec
S=\<S_\alpha\st\alpha<\lambda^\plus>$, with
$\Cof_\omega\lambda^\plus=\bigsqcup_{\alpha<\lambda^\plus}S_\alpha$.
Let $\vec T=j(\vec S)=\<T_\alpha\st\alpha<\lambda^\plus>$.

We claim for $\xi<\lambda^\plus$ that $\xi\in\ran(j)$ if
and only if $T_\xi$ is stationary in $V$. For the forward
direction, suppose that $\xi=j(\alpha)$. Let
$C=\set{\beta<\lambda^\plus\st j\image\beta\of\beta}$,
which is a club subset of $\lambda^\plus$ in $V$. Observe
that if $\beta\in C$ and $\cof(\beta)=\omega$, then
$j(\beta)=\beta$. Thus, if $\beta\in
C\intersect\Cof_\omega\lambda^\plus$, then $\beta\in
S_\alpha\iff\beta=j(\beta)\in j(S_\alpha)=T_{j(\alpha)}$.
Since $S_\alpha$ and $T_{j(\alpha)}$ are contained within
the ordinals of cofinality $\omega$, this means that
$C\intersect S_\alpha=C\intersect T_{j(\alpha)}$. In short,
$S_\alpha$ and $T_{j(\alpha)}$ agree on a club, and so
$T_{j(\alpha)}$ is stationary, as desired. Conversely,
suppose that $T_\xi$ is stationary in $V$. It follows that
there is some $\beta\in C\intersect T_\xi$. Since every
element of $T_\xi$ has cofinality $\omega$, it follows that
$j(\beta)=\beta$. But since
$\beta\in\Cof_\omega\lambda^\plus$, we must also have
$\beta\in S_\alpha$ for some $\alpha<\lambda^\plus$. It
follows that $\beta=j(\beta)\in j(S_\alpha)=T_{j(\alpha)}$.
So $T_\xi$ and $T_{j(\alpha)}$ have the element $\beta$ in
common, and since $\vec T$ is a partition, it must be that
$\xi=j(\alpha)$.

The claim of the previous paragraph shows that
$j\image\lambda^\plus$ is definable from $\vec T$, which is
in $\HOD$, and so $j\image\lambda^\plus\in\HOD$ as well.
From $j\image\lambda^\plus$, we can also define
$j\restrict\lambda^\plus$, and so $C\in\HOD$. To complete
the argument, let $S^*=T_\kappa$, which is a subset of
$\Cof_\omega\lambda^\plus$ that is stationary in \HOD.
Since $C$ is a club in \HOD, there is $\beta\in C\intersect
S^*$. Since $\cof(\beta)=\omega$ and $\beta\in C$, it
follows that $j(\beta)=\beta$. Since
$\beta\in\Cof_\omega\lambda^\plus$, it follows that
$\beta\in S_\alpha$ for some $\alpha<\lambda^\plus$. Thus,
$\beta=j(\beta)\in j(S_\alpha)=T_{j(\alpha)}$,
contradicting the fact that $T_\kappa$ and $T_{j(\alpha)}$
are disjoint.\end{proof}

Note that once we established $j \image \lambda^\plus\in
\HOD$, we could have finished the proof instead by using
Kunen's original method with $\omega$-J\'{o}nsson functions
or by using the technique of Harada \cite[p.
320-321]{Kanamori2004:TheHigherInfinite2ed} or of Zapletal
\cite{Zapletal1996:ANewProofOfKunenInconsistency}.

For any
class $A$, we define $\HOD(A)$ to be the class of sets that
are definable using ordinal parameters and parameters in
$A$. This differs somewhat from the more generous class
defined in \cite[p.195]{Jech:SetTheory3rdEdition}, where it
is also allowed that the class $A$ appears in the
definitions as a predicate, but as we do not need this
feature, our results are stronger with the more restrictive
notion.

\begin{theorem}\label{Theorem.Ifj:VtoMThenV=HOD(M)}
If $j:V\to M$ is an elementary embedding of $V$ into an
inner model $M$, then $V=\HOD(M)$. That is, every object in
$V$ is definable in $V$ using a parameter from $M$. Furthermore, every object in $V$ is definable using a parameter from the image of $j$.
\end{theorem}

\begin{proof}
This theorem can be viewed as a corollary to the proof of
theorem \ref{Theorem.Noj:VtoHOD} (although one could
reverse this and view theorem \ref{Theorem.Noj:VtoHOD} as a
corollary to this result, arguing that if $j:V\to\HOD$,
then $V=\HOD(\HOD)$, which means $V=\HOD$, which prevents
such a $j$). Suppose $j:V\to M$ is elementary. We shall
first prove that for every ordinal $\gamma$, the
restriction $j\restrict \gamma$ is definable in $V$ from
parameters in $M$. As in the proof of theorem
\ref{Theorem.Noj:VtoHOD}, we may find $\lambda>\gamma$ with
$j(\lambda)=\lambda$ and $\cof(\lambda)=\omega$, by
iterating the embedding $\omega$ many times above $\gamma$.
It follows that $j(\lambda^\plus)=\lambda^\plus$. And as in
that proof, we may again partition
$\Cof_\omega\lambda^\plus=\bigsqcup_{\alpha<\lambda^\plus}S_\alpha$
into stationary sets $\vec
S=\<S_\alpha\st\alpha<\lambda^\plus>$, and, setting $\vec{T} = j(\vec{S}) = \<T_\alpha \st \alpha<\lambda^\plus>$, observe for
$\xi<\lambda^\plus$ that $\xi\in\ran(j)$ if and only if
$T_\xi$ is stationary in $V$. Thus, $j\image\lambda^\plus$
and hence $j\restrict\lambda^\plus$ is definable in $V$
from $\vec T$, which is an element of $M$. In particular,
$j\restrict\gamma$ is definable from $\vec T$ and $\gamma$. This proof also shows that for an unbounded class of successor cardinals $\lambda^+$, $j \restrict \lambda^\plus$ is definable from parameters in the image of $j$, since $\vec{T}$ is in the range of $j$.

Since any set $A\of\lambda^\plus$ is easily definable from $j(A)$
and $j\image\lambda^+$, it follows that every set of ordinals
in $V$ is definable in $V$ from parameters in the range of $j$. Since
every set is coded by a set of ordinals, it follows that
every set in $V$ is definable in $V$ from parameters in the range of $j$, and in particular, $V=\HOD(M)$. The proof is thus complete.
\end{proof}

Theorem \ref{Theorem.Ifj:VtoMThenV=HOD(M)} shows in
particular for any set $A$ that $j\restrict A$ is definable
in $V$ using parameters from $M$.

\begin{corollary}\label{Corollary.Ifj:VtoMofHODthenV=HOD}
If $j:V\to M$ is an elementary embedding of $V$ into a
transitive class $M$ for which $M\of\HOD$, then $V=\HOD$.
\end{corollary}

\begin{proof}
Note that we need not assume \AC\ for this result, since
our hypothesis actually implies it: defining that $r$
precedes $s$ if and only if $j(r)$ precedes $j(s)$ in the
\HOD\ order pulls back the definable \HOD\ order from $M$
to a global well-ordering of $V$. The claim is now
immediate from theorem \ref{Theorem.Ifj:VtoMThenV=HOD(M)},
since if $M\of\HOD$, then $\HOD(M)=\HOD$, since having
parameters from $M$ doesn't help beyond having ordinal
parameters.
\end{proof}

In other words, corollary
\ref{Corollary.Ifj:VtoMofHODthenV=HOD} asserts that if
$V\neq\HOD$ and $M\of\HOD$, then there is no nontrivial
$j:V\to M$.

Using the ideas of theorem \ref{Theorem.Noj:VtoV[G]}, we
shall now improve theorem
\ref{Theorem.Ifj:VtoMThenV=HOD(M)} to the case of
embeddings $j:M\to N$, not necessarily defined on all of
$V$, provided that $M$ is eventually stationary correct.
This will allow us afterward to generalize the previous
corollaries to the case of the generic $\HOD$.

\begin{theorem}\label{theorem.Ifj:MtoNThenMofHOD(N)}
Suppose that $j:M\to N$ is a nontrivial elementary
embedding between inner models $M$ and $N$ of \ZFC\ and
that $M$ is eventually stationary correct to $V$. Then
$M\of\HOD(N)$ and furthermore $j\restrict A\in\HOD(N)$ for
any $A\in M$. Indeed, for every $A \in M$, $A$ and $j\restrict A$ are definable in V from parameters in the range of $j$.
\end{theorem}

\begin{proof}
Suppose $j:M\to N$ is elementary, where $M$ and $N$ are
inner models of \ZFC\ and $M$ is stationary correct to $V$
above $\theta$. Fix any ordinal $\gamma$. As in the proof
of theorem \ref{Theorem.NojMtoNifStationaryCorrect}, we may
find an ordinal $\lambda$ above both $\theta$ and $\gamma$
and of cofinality $\omega$ in $M$ such that
$j(\lambda)=\lambda$ and $\lambda>\theta$. Since $M$ is
stationary correct above $\theta$, it follows that
$\lambda^\plus$ and $\lambda^\plusplus$ are the same in $M$
and $V$, and from this it follows that
$j(\lambda^\plus)=\lambda^\plus$ and
$j(\lambda^\plusplus)=\lambda^\plusplus$, and so all three
models $M$, $N$ and $V$ agree on $\lambda^\plus$,
$\lambda^\plusplus$ and the set
$\Cof_{\lambda^\plus}\lambda^\plusplus$. By the Ulam-Solovay
theorem, we may partition
$\Cof_{\lambda^\plus}\lambda^\plusplus=\bigsqcup_{\alpha<\gamma}S_\alpha$
into stationary sets $\vec S=\<S_\alpha\st\alpha<\lambda^\plusplus>$
in $M$. Let $\vec{T} = \<T_\alpha \st \alpha < \lambda^\plusplus> = j(\vec{S}).$

We claim as in theorem \ref{Theorem.Ifj:VtoMThenV=HOD(M)}
for $\xi<\lambda^\plusplus$ that $\xi\in\ran(j)$ if and only if
$T_\xi$ is stationary in $V$. If $\xi=j(\alpha)$, then as
in the proof of theorem \ref{Theorem.Noj:VtoHOD}, we know
that $S_\alpha$ and $T_\xi$ agree on the club
$C=\set{\beta<\lambda^\plusplus \st j\image\beta\of\beta}$,
since any such $\beta$ of cofinality $\lambda^\plus$ is
fixed by $j$, and so $T_\xi$ is stationary in $V$.
Conversely, if $T_\xi$ is stationary in $V$, then there is
some $\beta\in T_\xi\intersect C$, and so
$\cof(\beta)=\lambda^\plus$ in $N$ and hence $V$ and $M$,
and so $\beta\in S_\alpha$ for some ordinal $\alpha < \lambda^\plusplus $. It follows that $\beta=j(\beta)\in
T_{j(\alpha)}$, which implies $\xi=j(\alpha)$ since the
sets in $\vec T$ are disjoint.

Thus, $j\image\lambda^\plusplus$ and hence $j\restrict\lambda^\plusplus$ is
definable in $V$ from $\vec T$, which is an element of the range of $j$.
Since any set $A\of\lambda^\plusplus$ in $M$ is easily definable from
$j(A)$ and $j\image\lambda^\plusplus$, it follows that every set of
ordinals in $M$ is definable in $V$ from parameters in $N$, indeed from parameters in the image of $j$.
Similarly, $j\restrict A$ is definable from $A$ and
$j\restrict\lambda^\plusplus$. Since every set in $M$ is coded by a
set of ordinals in $M$, it follows that every set $A$ in
$M$ is definable in $V$ from parameters in the image of $j$, so in particular, $M\of\HOD(N)$. By enumerating the set $A$, one may
similarly conclude that $j\restrict A$ is definable from parameters in the range of $j$ as desired.
\end{proof}

\begin{corollary}
There is no generic embedding $j:V\to\HOD$. That is, in no
set-forcing extension $V[G]$ is there a nontrivial
elementary embedding $j:V\to \HOD^V$.
\end{corollary}

\begin{proof}
Suppose that $j:V\to \HOD^V$ is a nontrivial elementary
embedding in a set-forcing extension $V[G]$. By performing
additional collapse forcing, we may assume that the forcing
is almost homogeneous and ordinal definable. By theorem
\ref{theorem.Ifj:MtoNThenMofHOD(N)}, it follows that every
element of $V$ is definable in $V[G]$ using parameters in
$\HOD^V$, that is, $V\of\HOD(\HOD^V)^{V[G]}$. But
$\HOD(p)^{V[G]}\of\HOD(p)^V$ for any parameter $p\in V$, by
our assumption on the forcing, and so
$\HOD(\HOD^V)^{V[G]}=\HOD^V$. So we've argued that
$V\of\HOD^V$ and hence $V=\HOD^V$. The nonexistence of $j$
now follows from corollary \ref{Corollary.NoGenericjVtoV}.
\end{proof}

Let us now explain how this analysis extends to the case of
the iterated $\HOD$ classes $\HOD^\eta$, obtained by
iteratively relativizing the $\HOD$ class. To define these
classes, we begin with $\HOD^0=V$ and define
$\HOD^{n+1}=\HOD^{\HOD^n}$, so that $\HOD^1$ is simply
$\HOD$ and $\HOD^2=\HOD^\HOD$ is \HOD\ as computed inside
\HOD, and so on. We would naturally want to continue this
iteration with an intersection
$\HOD^\omega=\Intersect_{n<\omega}\HOD^n$ at $\omega$, but
for a subtle metamathematical reason, this may not succeed
in defining a class. The reason is that our previous
definition of $\HOD^n$ was by a {\it meta-theoretic}
induction on $n$; although each $\HOD^n$ is a definable
class for any meta-theoretic natural number $n$, these
definitions become progressively more complex as $n$
increases, and the definition does not provide a uniform
presentation of the $\HOD^n$, which we seem to need in
order to take the intersection $\Intersect_n\HOD^n$.
Indeed, a 1974 result of Harrington appearing in
\cite[section
7]{Zadrozny1983:IteratingOrdinalDefinability}, with related
work in \cite{McAloon1974:OnTheSequenceHODn}, shows that it
is consistent with \NGBC\ that $\HOD^n$ exists for each
$n<\omega$ but the intersection $\HOD^\omega$ is not a
class. Nevertheless, some models have a special structure
allowing them to enjoy a uniform definition of these
classes, and in these models we may continue the iteration
transfinitely. To illuminate this situation, we define that
a class $H$ is a {\df uniform presentation} of
$\HOD^\alpha$ for $\alpha<\eta$ if $H\of \set{(x,\alpha)\st
\alpha<\eta}$ and the slices $H^\alpha=\set{x\st
(x,\alpha)\in H}$ for $\alpha<\eta$ are all models of \ZF\
and obey the defining properties of $\HOD^\alpha$, namely,
the base case $H^0=V$, the successor case
$H^{\alpha+1}=\HOD^{H^\alpha}$ and the limit case
$H^\gamma=\Intersect_{\alpha<\gamma}H^\alpha$ for limit
ordinals $\gamma$. By induction, any two such classes $H$
agree on their common coordinates, and we shall write
$\HOD^\alpha$ to mean $H^\alpha$ for such a class $H$,
which will be fixed in the background for a given
discourse. Let us define the phrase ``{\df $\HOD^\eta$
exists}'' to mean that $\eta$ is an ordinal and there is a
uniform presentation of $\HOD^\alpha$ for $\alpha\leq\eta$.
This is nearly equivalent to the existence of a uniform
presentation of $\HOD^\alpha$ for $\alpha<\eta$, since the
missing class $\HOD^\eta$ can be computed from that
information: if $\eta$ is a limit ordinal, then $\HOD^\eta$
is the intersection of all $\HOD^\alpha$ for $\alpha<\eta$,
and if this is a \ZF\ model, then we can also say that
$\HOD^\eta$ exists; and if $\eta=\beta+1$ is a successor
ordinal, then $\HOD^\eta=\HOD^{\HOD^\beta}$, so $\HOD^\eta$
exists. It is easy to see that $\HOD^n$ exists for any
(meta-theoretic) natural number $n$, and if $\HOD^\eta$
exists, so does $\HOD^{\eta+1}$ and $\HOD^\alpha$ for any
$\alpha<\eta$. But by the Harrington result we mentioned
earlier, one cannot necessarily proceed through limit
ordinals or even up to $\omega$ uniformly. Note that even
when $\HOD^\eta$ exists, there seems little reason to
expect that it is necessarily a definable class, even when
$\eta$ is definable or comparatively small, such as
$\eta=\omega$, although in most of the cases where we know
$\HOD^\eta$ exists it is because we do in fact have a
uniform definition. We now generalize the Kunen
inconsistency and theorem \ref{Theorem.Noj:VtoHOD} to the
case of embeddings from $V$ to any $\HOD^\eta$ or its
eventually stationary correct submodels.

\begin{corollary}\label{Corollary.nojVtoHODeta}
If\/ $\HOD^\eta$ exists, then there is no nontrivial
elementary embedding $j:V\to\HOD^\eta$. More generally,
if\/ $\HOD^\eta$ exists and $M\of\HOD^\eta$ is eventually
stationary correct relative to $\HOD^\eta$, then there is
no nontrivial elementary embedding $j:V\to M$. Indeed, no
such embeddings exist in any set-forcing extension $V[G]$.
\end{corollary}

\begin{proof}
For the easy case, if $j:V\to\HOD^\eta$, then since
$\HOD^\eta\of\HOD$, it follows by corollary
\ref{Corollary.Ifj:VtoMofHODthenV=HOD} that $V=\HOD$ and
hence $V=\HOD^\eta$. The existence of such a $j$ is now
ruled out by theorem
\ref{Theorem.NojMtoNifStationaryCorrect}. More generally,
suppose that $j:V\to M\of\HOD^\eta$ is a nontrivial
elementary embedding that is a class in some set-forcing
extension $V[G]$, and that $M$ is eventually stationary
correct to $\HOD^\eta$. We may assume by further collapse
forcing if necessary that the forcing giving rise to $V[G]$
is ordinal definable and almost homogeneous. Since $V$ is
eventually stationary correct to $V[G]$, it follows by
theorem \ref{theorem.Ifj:MtoNThenMofHOD(N)} applied to $j$
in $V[G]$ that every element of $V$ is definable in $V[G]$
from parameters in $M$, which are all in $\HOD^\eta$, and
so $V\of\HOD^V$, which implies $V=\HOD^\eta$. The existence
of $j$ is now ruled out by theorem
\ref{Theorem.NojMtoNifStationaryCorrect} applied in $V[G]$, since $M$ is eventually stationary correct in $\HOD^\eta$.
\end{proof}

We should now like to generalize some of the previous
results from $\HOD$ to the case of the {\df generic \HOD},
denoted $\gHOD$, defined to be the intersection of the
$\HOD$s of all set-forcing extensions. It suffices to
consider only forcing notions of the form $\Coll(\omega,\theta)$,
as these absorb all other forcing and do so while not
enlarging $\HOD$, because they are almost-homogeneous and
ordinal definable. The $\gHOD$ is a definable transitive
proper class model of \ZFC\ and invariant by set forcing.
The generic $\HOD$ was introduced by Gunter Fuchs
\cite[p.~298]{Fuchs2008:ClosedMaximalityPrinciples} and
further explored in
\cite{FuchsHamkinsReitz:Set-theoreticGeology}, where
results show that it is consistent that the $\gHOD$ is far
smaller than $\HOD$ and also smaller than the mantle, the
intersection of all set-forcing ground models of $V$. For any class
$A$, the class $\HOD(A)$ consists of the sets that are
hereditarily definable using ordinal parameters or
parameters in $A$. The class $\gHOD(A)$ is the intersection
of $\HOD(A)^{V[G]}$ over all set-forcing extensions $V[G]$.

\begin{corollary}
Suppose that $j:M\to N$ is a nontrivial elementary
embedding between inner models $M$ and $N$ of \ZFC\ and
that $M$ is eventually stationary correct to $V$. Then
$M\of\gHOD(N)$ and furthermore $j\restrict A\in \gHOD(N)$
for any $A\in M$.
\end{corollary}

\begin{proof}
If $j:M\to N$ is a nontrivial elementary embedding between
inner models $M$ and $N$ and $M$ is eventually stationary
correct to $V$, then it is also eventually stationary
correct to $V[G]$ for any set-forcing extension of $V$, and
so by theorem \ref{theorem.Ifj:MtoNThenMofHOD(N)} applied
in $V[G]$, it follows that $M\of\HOD(N)^{V[G]}$. But
$\gHOD(N)$ in $V$ is the intersection of all such
$\HOD(N)^{V[G]}$, so the proof is complete.
\end{proof}

\begin{corollary}\label{Corollary.Ifj:VtoMofgHODThenV=gHOD}
If $j:V\to M$ and $M\of\gHOD$, then $V=\gHOD$.
\end{corollary}

\begin{proof}
Since $\gHOD\of\HOD$ and this remains true in any
set-forcing extension, and since $V$ is eventually
stationary correct to every set-forcing extension V[G], it
follows by theorem \ref{theorem.Ifj:MtoNThenMofHOD(N)} that
$V\of\HOD^{V[G]}$ for all such extensions, and this implies
$V\of\gHOD$ and hence $V=\gHOD$, as desired.
\end{proof}

The ideas apply to the iterated $\gHOD$ construction as
well. We use the phrase ``{\df $\gHOD^\eta$ exists}''
to mean that $\eta$ is an ordinal and there is a class $H$,
whose slices $H^\alpha$ are all models of \ZF\ and
constitute a uniform presentation of the $\gHOD^\alpha$ for
$\alpha \leq \eta$, having the correct base case $H^0=V$,
successor step $H^{\alpha+1}=\gHOD^{H^\alpha}$ and limit
case $H^\gamma=\Intersect_{\alpha<\gamma}H^\alpha$ (for
limit ordinals $\gamma$).

\begin{corollary}\label{Corollary.nojVtogHODeta}
There is no nontrivial elementary embedding $$j:V\to
\gHOD.$$ If\/ $\gHOD^\eta$ exists, then there is no
nontrivial elementary embedding
$$j:V\to\gHOD^\eta.$$ More generally, if\/ $\gHOD^\eta$
exists and $M\of\gHOD^\eta$ is eventually stationary
correct relative to $\gHOD^\eta$, then there is no
nontrivial elementary embedding $j:V\to M$.
\end{corollary}

\begin{proof}
If $j:V\to M\of\gHOD^\eta$, then since $M\of\gHOD$ it
follows by corollary
\ref{Corollary.Ifj:VtoMofgHODThenV=gHOD} that $V=\gHOD$ and
so $V=\gHOD^\eta$. The embedding is now ruled out by
theorem \ref{Theorem.NojMtoNifStationaryCorrect}.
\end{proof}

We note as before that since \AC\ holds automatically in
$\gHOD$, we need not assume \AC\ in $V$ when ruling out a
nontrivial elementary embedding $j:V\to\gHOD$, as this
assumption follows by elementarity. A similar observation
holds for $j:V\to\gHOD^{\eta+1}$.

The results of corollaries \ref{Corollary.nojVtoHODeta} and
\ref{Corollary.nojVtogHODeta} generalize to the case of
$\HOD[A]$ and $\gHOD[A]$, where $A$ is any class, as well
to as their iterates and eventually stationary correct
submodels. Namely, the class $\HOD[A]$ is the class of all
sets hereditarily definable from ordinal parameters and
using $A$ as a predicate, and $\<\HOD[A],\in,A>$ is a model
of $\ZFC(A)$. Iterating this idea, we define that
$\HOD[A]^\eta$ exists if there is uniform presentation
class $H$ whose slices obey the desired defining
properties, so that $H^0=V$,
$H^{\alpha+1}=\HOD[A]^{H^\alpha}$ and
$H^\gamma=\Intersect_{\alpha\lt\gamma}H^\alpha$ for limit
$\gamma$. A similar definition applies to $\gHOD[A]$ and
$\gHOD[A]^\eta$.

\begin{corollary}\label{Corollary.nojVtoHOD[x]eta}
If $A$ is any class and $\HOD[A]^\eta$ exists, then there
is no nontrivial elementary embedding $j:V\to\HOD[A]^\eta$,
and indeed, no nontrivial elementary embedding $j:V\to M$ for any class $M\of\HOD[A]^\eta$
that is eventually stationary correct with respect to
$\HOD[A]^\eta$. Similarly, if $\gHOD[A]^\eta$ exists, then
there is no nontrivial elementary embedding $j:V\to M$ for
any class $M\of\gHOD[A]^\eta$ that is eventually stationary
correct in $\gHOD[A]^\eta$. And no such embeddings can be
found in any set-forcing extension $V[G]$.
\end{corollary}

The proof proceeds simply by carrying the class $A$ all the
way through the earlier proofs. For example, if
$j:V\to\HOD[A]$, then it follows by theorem
\ref{Theorem.Ifj:VtoMThenV=HOD(M)} that every element of
$V$ is definable using parameters from $\HOD[A]$, and so
$V=\HOD[A]$, which now violates theorem
\ref{Theorem.NojMtoNifStationaryCorrect}.

The results in this section all rule out nontrivial $j:V\to
M$ for various definable classes $M$, usually with
$M\of\HOD$. Perhaps it is tempting to hope for a completely
general result ruling out $j:V\to M$ for any definable
class $M$, or perhaps for any definable class $M\of \HOD$.
But alas, such a result is not generally true, assuming the
consistency of a measurable cardinal. To see this, observe
that in the canonical inner model $V=L[\mu]$ with one
measurable cardinal, there is a unique normal measure $\mu$
on a unique measurable cardinal. The ultrapower
$j:L[\mu]\to L[j(\mu)]$ by this normal measure is
consequently a parameter-free definable embedding of
$V=L[\mu]$ into a definable class $L[j(\mu)]$, and the
model satisfies $V=\HOD$. So we should not hope to rule out
all nontrivial $j:V\to M$ into a definable class
$M\of\HOD$. Nevertheless, in theorem
\ref{Theorem.Noj:definableMtoV} we shall rule out
nontrivial embeddings in the converse direction $j:M\to V$,
from any definable class $M$ to $V$.

\section{Embeddings from definable class models into V} \label{Section.jfromMofHODtoV}

In this section, we rule out elementary embeddings in the
converse direction, from $\HOD\to V$ and from other
definable classes into $V$, and the arguments will have a
very different character than theorem
\ref{Theorem.Noj:VtoHOD}. In particular, the argument
ruling out $j:\HOD\to V$ will not rely on any result in
infinite combinatorics, such as the stationary partition
theorem. Instead, we shall extend the embedding $\HOD\to V$
into an infinite inverse system of embeddings
$$\cdots\rTo \HOD^n\rTo\cdots\rTo\HOD^2\rTo\HOD\rTo
V,$$ and then analyze the nature of the inverse limit,
ultimately relying on the fact that $\HOD$ has a definable
well-order in $V$. The overall argument is soft, simply
making use of the inverse system, but the details run into
a subtle metamathematical issue, which we resolve,
concerning the extent to which we may treat the inverse
system uniformly in $n$, since as we have mentioned the
$\HOD^n$ sequence is not generally uniformly definable.

After proving in theorem \ref{Theorem.NojHODtoV} that there
is no nontrivial elementary embedding from $\HOD$ to $V$,
we mount a generalization of the methods, culminating in
theorem \ref{Theorem.Noj:definableMtoV}, asserting that if
$M$ is any definable class, then there is no nontrivial
elementary embedding $j:M\to V$. Few of the theorems in
this section require the axiom of choice, although in
several cases this is simply because \AC\ follows from the
other assumptions.

We begin with the observation that there is a very easy
proof of the Kunen inconsistency in the case $V=\HOD$.

\begin{theorem} \label{Theorem.NojVtoV=HOD}
If\/ $V=\HOD$, then there is no nontrivial elementary
embedding $j:V\to V$.
\end{theorem}

\begin{proof}
Let $\kappa$ be the critical point of $j$, and define the
usual critical sequence $\kappa_{n+1}=j(\kappa_n)$, with
$\kappa_0=\kappa$ and $\lambda=\sup\<\kappa_n\st
n<\omega>$. By applying $j$ to the critical sequence, it
follows easily that $j(\lambda)=\lambda$. Since $V=\HOD$,
there is a $\HOD$-least $\omega$-sequence $s=\<\alpha_n\st
n<\omega>$ such that $\lambda=\sup\<\alpha_n\st n<\omega>$.
Since $s$ is definable from $\lambda$ and
$j(\lambda)=\lambda$, it must be that $j(s)=s$. In
particular, $j(\alpha_n)=\alpha_n$ for every $n<\omega$.
But there are no fixed points of $j$ between $\kappa$ and
$\lambda$, so $s$ is not cofinal in $\lambda$ after all, a
contradiction.\end{proof}

We suggest that this argument may be the simplest proof
that there is no nontrivial $j:V\to L$. Also, we note that
one doesn't need the full $V=\HOD$ in theorem
\ref{Theorem.NojVtoV=HOD}, but only a definable
well-ordering of $[\lambda]^\omega$, and this observation
allows us to transfer the argument to the $I_1(\kappa)$
context.

\begin{corollary} Assume only \ZF\ and suppose that $j:V_{\lambda+1}\to V_{\lambda+1}$
is a nontrivial elementary embedding. Then there is no
well-ordering of $[\lambda]^\omega$ definable in
$V_{\lambda+1}$.
\end{corollary}

\begin{proof}
The same argument as in theorem \ref{Theorem.NojVtoV=HOD}
works here, except that we choose $s$ by using the
definable well-ordering of $[\lambda]^\omega$.
\end{proof}

Let us now rule out the possibility of a nontrivial
elementary embedding $j:\HOD\to V$. After this, we shall
generalize to other inner models.

\begin{theorem}\label{Theorem.NojHODtoV}
There is no nontrivial elementary embedding
 $$j:HOD \to V.$$
\end{theorem}

\begin{proof}
Assume to the contrary that there is such a nontrivial
elementary embedding $j:\HOD \to V$. As before, we needn't
assume \AC\ explicitly since our hypothesis implies that
\AC\ holds in $V$, since it holds in \HOD\ and $j$ is
elementary. We begin by constructing from $j$ a uniform
presentation of the classes $\HOD^n$ for $n<\omega$. In
order to do so, we would like to iterate $j$, and we remark
that for this kind of embedding, where the domain is a
proper subclass of the codomain, one does not iterate the
embedding in the usual forward direction, since it can
happen for a set $x$ that $j(x)$ is no longer in the domain
of $j$, leaving $j^2(x)$ undefined. Rather, one should
build an inverse system, iterating the embedding in the
reverse direction, with the domains becoming successively
smaller. In order to do this, define that $x$ is in the
domain of $j^n$ if it is possible to apply $j$ successively
$n$ times to $x$; that is, if there is a sequence
$\<x_0,\ldots,x_n>$ such that $x_0=x$ and $x_{n+1}=j(x_n)$.
It follows that $x\in\dom(j^{n+1})\iff j(x)\in\dom(j^n)$,
where we consider $j^0$ as the universal identity function.
Let $H$ be the class $\set{(x,n)\st x\in \dom(j^n)}$. In
short, we define a class $H$ by specifying its slices as
$H^n=\dom(j^n)$. We will show that $H$ is a uniform
presentation of the $\HOD^n=H^n$. Clearly, $H^0=V$ and
$H^1=\HOD$, and we are on our way; it remains to prove
$H^{n+1}=\HOD^{H^n}$, which we now do by induction. This
argument makes subtle but critical use of the fact, due to
lemmas \ref{Lemma.Gaifman} and
\ref{Lemma.ZFFirstOrderExpressible}, that the statements of
the following claim are each expressible by single \NGBC\
assertions in the class parameter $j$, with natural number
parameter $n$.\footnote{The subtle point is that one cannot
generally prove a {\it scheme} of statements $\varphi_i(n)$
by induction on $n$, since if the scheme is not
expressible, we are not able in \NGBC\ to form the set of
$n$ where it fails and thus may be unable to find the least
$n$ where it fails.}
\begin{subclaim} \label{Claim.UniformHODnDefinitionEtc}
For all $n \in \omega$,
\begin{enumerate}
 \item $H^{n+1}=\HOD^{H^n}$.
     \label{Statement.HODnUniformlyDefinable}
 \item $H^n$ and $H^{n+1}$ are transitive, proper class
     models of ZFC. \label{Statement.H^n+1ZFCModel}
 \item $j\restrict H^{n+1}:H^{n+1}\to H^n$ is
     elementary.\label{Statement.Elementary}
\end{enumerate}
\end{subclaim}
\begin{proof}
Note that we need statement \ref{Statement.H^n+1ZFCModel}
to hold even in order for statement
\ref{Statement.Elementary} to be first-order expressible,
since Gaifman's lemma \ref{Lemma.Gaifman} as we have stated
it applies only to transitive models of \ZF.\footnote{For
example, recent work of Gitman, Hamkins and Johnstone
\cite{GitmanHamkinsJohnstone:WhatIsTheTheoryZFC-Powerset?}
shows that Gaifman's theorem can fail for transitive proper
class models of the version of \ZFC\ without Powerset using
Replacement without Collection, although it does apply to
$\ZFC^-$ models when this theory includes Collection.} The
case $n=0$ is immediate. Assume inductively that the claim
holds for $n$, and consider $n+1$. We start with statement
\ref{Statement.HODnUniformlyDefinable}. For all $x \in
H^{n+1}$, we have
\begin{align*}
&x \in \HOD^{H^{n+1}}\\
&\iff H^{n+1} \satisfies x\in\HOD\\
&\iff H^n \satisfies j(x)\in\HOD &&\text{by inductive hypothesis \ref{Statement.Elementary} }\\
&\iff j(x) \in H^{n+1} &&\text{by inductive hypothesis \ref{Statement.HODnUniformlyDefinable}} \\
&\iff x \in H^{n+2} &&\text{by definition of $H$}.
\end{align*}
So statement \ref{Statement.HODnUniformlyDefinable} is
proven for $n+1$.

Next, we prove statement \ref{Statement.H^n+1ZFCModel}. By
the inductive hypothesis, $H^{n+1}$ is a \ZFC\ model and
therefore satisfies the sentence stating that \HOD\ is a
transitive, proper class \ZFC\ model. Since we just proved
that $H^{n+2}=\HOD^{H^{n+1}},$ it follows that $H^{n+2}$ is
a transitive, proper class \ZFC\ model.

Finally, we prove statement \ref{Statement.Elementary}.
First, note by the definition of $H$ that $j$ maps
$H^{n+2}$ into $H^{n+1}$. Since $H^{n+2}$ is the \HOD\ of
$H^{n+1}$ and $H^{n+1}$ is the \HOD\ of $H^n$, and we know
by the inductive hypothesis that $j\restrict
H^{n+1}:H^{n+1}\to H^n$ is elementary, it follows that
$j\restrict H^{n+2}:H^{n+2}\to H^{n+1}$ is elementary,
since this is the restriction of an elementary embedding to
the definable transitive class \HOD\ of the domain and its
corresponding codomain. So statement
\ref{Statement.Elementary} holds and the proof of claim
\ref{Claim.UniformHODnDefinitionEtc} is complete.
\end{proof}

We may now freely refer to $\HOD^n$ uniformly in $n$. The
claim shows moreover that we have an entire inverse system
of embeddings
$$\cdots\rTo \HOD^n\rTo\cdots\rTo\HOD^2\rTo\HOD\rTo
V,$$ where the embedding at each step is the appropriate
restriction of $j$. By composing, $j^n:\HOD^n\to V$ is
elementary.

Let $<^{n+1}$ be the canonical well-ordering of
$\HOD^{n+1}$ definable in $\HOD^n$. Let $<^0 = j(<^1)$, in
the sense of applying $j$ to a class, meaning $<^0 =
\Union_{\alpha \in \ORD }j(<^1 \intersect V_\alpha)$. Note
that our uniform presentation of the $\HOD^n$'s allows us
to define the $<^{n+1}$ uniformly in $n$. Since $j$ is
elementary from each $\HOD^n$ to $\HOD^{n-1}$, it follows
that $j(<^{n+1})=<^n$ for all $n<\omega$, again in the
sense of applying $j$ to a class.

The key concept of this proof is the definition that a
sequence $\vec x=\<x_n\st n<\omega>$ is {\df $j$-coherent},
if $j(x_{n+1})=x_n$ for all $n<\omega$. Such sequences
arise naturally in the inverse limit of the system of
embeddings above. We shall now derive a contradiction by
showing first, that every $j$-coherent sequence is
constant, and second, that there is a $j$-coherent sequence
that is not constant.

We show first that every $j$-coherent sequence is constant,
having the form $\<x,x,x,\ldots>$ for a fixed point
$x=j(x)$. To see this, suppose that $\vec x=\<x_n\st
n<\omega>$ is a nonconstant $j$-coherent sequence, where
$x_0$ has minimal possible $\in$-rank, which we denote
$\rk(x_0)$. Since $j(\rk(x_{n+1}))=\rk(x_n)$ by
elementarity, it follows that $\rk(x_{n+1})\leq\rk(x_n)$.
If $\rk(x_1)<\rk(x_0)$, then since the sequence is
$j$-coherent, we may apply the inverse of $j^n$ and see
that $\rk(x_{n+1})<\rk(x_n)$, thereby producing an infinite
descending sequence of ordinals, which is impossible. So it
must be that all the $x_n$ have the same rank. It is
similarly easy to see that $x_{n+1}\neq x_n$, since
otherwise the entire sequence would be constant, contrary
to our assumption. Let $a_n$ be the $<^n$-least element of
the symmetric difference $x_{n+1}\vartriangle x_n$, which
makes sense because these two sets are distinct elements of
$\HOD^n$. By the $j$-coherence of the well-orderings $<^n$
and of $\vec x$, it follows that $j(a_{n+1})=a_n$, and so
$\vec a=\<a_n\st n<\omega>$ is $j$-coherent. Note that
$a_1\neq a_0$, since if $a_0\in x_1\setminus x_0$ then
$a_1\in x_2\setminus x_1$ by elementarity, and otherwise
$a_0\in x_0\setminus x_1$, leading to $a_1\in x_1\setminus
x_2$, which makes $a_0=a_1$ impossible in either case.
Finally, since $a_0\in x_0\vartriangle x_1$, the rank of
$a_0$ is smaller than the rank of $x_0$ (which is equal to
the rank of $x_1$). This contradicts our assumption that
$\vec x$ was a minimal counterexample, and so we have
established that all $j$-coherent sequences are trivial.

We now obtain a contradiction by producing a nontrivial
$j$-coherent sequence. By theorem
\ref{Theorem.NojVtoV=HOD}, we know that $\HOD \ofnoteq V$,
and consequently $\HOD^{n+1}\ofnoteq \HOD^n$ by
elementarity. Let $s_n$ be the $<^n$-least element of
$\HOD^n\setminus \HOD^{n+1}$. It follows by the
$j$-coherence of the relations $<^n$ that $j(s_{n+1})=s_n$,
and so this sequence is $j$-coherent. Since $s_0\in
V\setminus \HOD$ and $s_1\in \HOD \setminus \HOD^2$, it
follows that $s_0\neq s_1$, and so this sequence is not
constant, a contradiction.\end{proof}

We shall now generalize theorem \ref{Theorem.NojHODtoV} to
other definable classes by first establishing the following
fundamental fact. A class $A$ is {\df $b$-definable} in a
transitive inner model $M$ if there is a formula $\varphi$
such that $A=\set{x\in M\st M\satisfies\varphi(x,b)}$. If
$N$ is another model containing $b$, we may relativize the
definition to obtain $A^N=\set{x\in N\st
N\satisfies\varphi(x,b)}$. The transitive closure of a
class is the smallest transitive class containing all
elements of that class.

\begin{theorem}\label{Theorem.A^M=A^N}
Do not assume \AC. Suppose that $j:M\to N$ is an elementary
embedding of inner models $M\of N$ of \ZF, where $M$ is
$b$-definable in $N$ with a parameter $b$ fixed by $j$.
Suppose that $A$ is a $b$-definable class (or set) in $N$
and that the transitive closure of $A$ has a $b$-definable
well ordering in $N$; or equivalently, $A$ is a $b$-definable class in $N$, and $A \of \HOD[b]^N$. Then $A^M=A^N$.
\end{theorem}

\begin{proof}
First, we note that given $A$ is $b$-definable in $N$, the inclusion $A \of \HOD[b]^N$ is equivalent to the existence of a $b$-definable well-ordering of $\tcl(A)$ in $N$. In one direction, if $A \of \HOD[b]^N$, then the canonical $HOD[b]^N$ ordering orders $\tcl(A)$. Conversely, if $\tcl(A)$ has a $b$-definable well-ordering, then since $A$ is $b$-definable, it follows for each ordinal $\alpha$ that the $\alpha$-th element of $A$ under this well-ordering is definable from $b$ and $\alpha$ in $N$, and so $A \of \HOD[b]^N$.

The theorem is formalized as an \NGB\ theorem scheme, asserting of
each pair of formulas $\psi$ and $\varphi$ that if they
define $M$ and $A$ respectively in $N$ in the way described, then the
conclusion holds. To begin the proof, suppose that $j:M\to
N$ is as described in the hypothesis of the theorem. Let
$M^0=N$ and $M^n = \dom(j^n)$, as defined in the proof of
theorem \ref{Theorem.NojHODtoV}, so that we have a uniform
presentation of these classes. In this notation, the
original embedding is $j:M^1\to M^0$. Note that because
$b=j(b)$, it follows that $b\in M^n$ for every $n$.

\begin{subclaim}\label{Claim.uniformMDef}
For all $n \in \omega$,
\begin{enumerate}
 \item $M^{n+1}$ is the $M$ of $M^n$, using the same
     definition of $M$ as in $N$, with the same
     parameter $b$. \label{Hypothesis.M}
 \item $M^n$ and $M^{n+1}$ are transitive, proper class
     models of \ZF.\label{Hypothesis.MnZFModel}
 \item $j\restrict M^{n+1}:M^{n+1}\to M^n$ is
     elementary.\label{Hypothesis.ElementaryM}
\end{enumerate}
\end{subclaim}
\noindent The proof is by induction on $n$, and follows the
proof of claim \ref{Claim.UniformHODnDefinitionEtc} by
substituting $M$ in place of $\HOD$ and using the fact that
$j(b) = b$. Thus, the definable classes $M^n$ essentially
form a $j$-coherent sequence, and we omit the details.
%
%
The claim leads to the inverse system
$$\cdots\rTo M^n\rTo\cdots\rTo M^2\rTo M^1 \rTo
M^0,$$ where the embedding at each step is the appropriate
restriction of $j$, and the final step is the full original
embedding $j:M\to N$.

Let us denote by $<^0$ the hypothesized $b$-definable
well-order of the transitive closure $\tcl(A)$ in $N$. For
each natural number $n$, let $<^n$ be the corresponding
well-ordering of $\tcl(A^{M^n})$ defined in $M^n$ using the
same formula and parameter $b$. The fact that this
definition is indeed a well-order of $\tcl(A^{M^n})$ in
$M^n$ follows by the elementarity of $j$ and the fact that
$j(b)=b$. Note that the $<^{n}$ can be uniformly presented
with respect to $n$, using our uniform presentation of the
classes $M^n$. Furthermore, since the $<^n$ are all defined
in $M^n$ by the same formula and the parameter is fixed by
$j$, it follows by elementarity that $j(<^{n+1})=<^n$, in
the sense of applying $j$ to a class.

Using the $j$-coherent concept of theorem
\ref{Theorem.NojHODtoV}, we show that every $j$-coherent
sequence $\vec{x} = \< x_n \st n \in \omega >$ with $x_0\in
\tcl(A^{M^0})$ is constant. If not, then let $\vec{x}$ be a
nonconstant $j$-coherent sequence with $x_0\in
\tcl(A^{M^0})$, such that $x_0$ has minimal rank among all
such sequences. By elementarity, it follows that $x_n \in
\tcl(A^{M^n})$ for every $n<\omega$. Since $\vec x$ is not
constant, we must have $x_0\neq x_1$. Suppose, for a first
case, that $x_0\setminus x_1$ is nonempty. It follows by
elementarity that $x_n\setminus x_{n+1}$ is nonempty for
all $n$, and we may let $a_n$ be the $<^n$-least element of
$x_n\setminus x_{n+1}$. Note that $a_n \in \tcl(A^{M^n})$.
By the $j$-coherence of $<^n$ and $\vec x$, it follows that
$j(a_{n+1})=a_n$, and so $\vec a=\<a_n\st n<\omega>$ is
$j$-coherent. Note that $\vec a$ is not constant, since
$a_0\in x_0\setminus x_1$ and $a_1\in x_1\setminus x_2$,
and $a_0\in x_0$ implies that $a_0$ is in $\tcl(A^{M^0})$
with strictly lower rank than $x_0$. The existence of $\vec
a$ therefore contradicts our minimality assumption on $\vec
x$ and $x_0$. For the remaining case, $x_1\setminus x_0$ is
nonempty. By elementarity, $x_{n+1}\setminus x_n$ is
nonempty for every $n$, and we may let $a_{n+1}$ be the
$<^{n+1}$-least element of $x_{n+1}\setminus x_n$ and also
define $a_0=j(a_1)$. By the uniformity of the definition,
it again follows that $j(a_{n+1})=a_n$, and so $\vec
a=\<a_n\st n<\omega>$ is $j$-coherent. It is not constant
since $a_1\in x_1\setminus x_0$, whilst $a_2\in
x_2\setminus x_1$, and from $a_1\in x_1$ we deduce that
$a_0\in x_0\in \tcl(A^{M^0})$ and so $a_0$ is in
$\tcl(A^{M^0})$ with strictly lower rank than $x_0$, again
contradicting our minimality assumption.

Finally, under the assumption that $A^M \neq A^N$, we
construct a nonconstant $j$-coherent sequence, $\vec{s} =
\<s_n \st n<\omega>$, with $s_0 \in \tcl(A^{M^0})$, thereby
contradicting the fact just established that all such
sequences are constant. If it happens that
$A^{M^0}\setminus A^{M^1}$ is nonempty, then let $s_0$ be
the $<^0$-least element of this difference class. Since
this is definable in $M^0$ from parameter $b=j(b)$, it
extends naturally to a $j$-coherent sequence $\vec
s=\<s_n\st n<\omega>$, where $s_n$ is the $<^n$-least
element of $A^{M^n}\setminus A^{M^{n+1}}$. This sequence is
not constant because $s_0\in A^{M^0}\setminus A^{M^1}$ but
$s_1\in A^{M^1}\setminus A^{M^2}$, and it has $s_0\in
A^{M^0}$, as desired. In the remaining case,
$A^{M^1}\setminus A^{M^0}$ is nonempty. Let $s_{n+1}$ be
the $<^{n+1}$-least element of $A^{M^{n+1}}\setminus
A^{M^n}$, and let $s_0=j(s_1)$. By the uniformity of these
definitions, it follows that $j(s_{n+1})=s_n$, and so $\vec
s=\<s_n\st n<\omega>$ is $j$-coherent. But it is not
constant, because $s_1\in A^{M^1} \setminus A^{M^0}$ whilst
$s_2\in A^{M^2}\setminus A^{M^1}$, and since $s_1\in
A^{M^1}$, we have $s_0\in A^{M_0} \of \tcl(A^{M^0})$. So
once again, we have contradicted the fact established in
the previous paragraph, and the proof is complete.
\end{proof}

\begin{corollary}\label{Corollary.SameCardCofEtc.}
If $j:M\to N$ is an elementary embedding of inner models
$M\of N$ of \ZF, where $M$ is $b$-definable in $N$ using a
parameter $b$ fixed by $j$, then $M$ and $N$ have
\begin{enumerate}
  \item the same cardinals and the same cofinality
      function,
  \item the same continuum function,
  \item the same $\Diamond^*_\kappa$ pattern and
  \item the same large cardinals of any particular
      kind.
  \item Furthermore, $\HOD^M=\HOD^N$ and
      $\gHOD^M=\gHOD^N$ and more.
\end{enumerate}
\end{corollary}

\begin{proof}
As in theorem \ref{Theorem.A^M=A^N}, we do not need to
assume \AC. The corollary follows immediately from the
theorem, because in each case we have a definable class
having a definable well-order on its transitive closure.
For example, if $A$ is the class of cardinals in $N$, then
this is definable in $N$ and the transitive closure is the
class $\ORD^N$, which certainly has a definable well-order
in $N$. So by the theorem, $A^M=A^N$, and so $M$ and $N$
have the same cardinals. Similarly, we can let $A$ be the
graph of the cofinality function, or the graph of the
continuum function $\gamma\mapsto 2^\gamma$, or the class
of cardinals $\kappa$ for which $\Diamond^*_\kappa$, or the
class of measurable cardinals, or the class of supercompact
cardinals or what have you. In each case, the class is
definable and has a definable well-order on the transitive
closure, and so the theorem implies that the class has the
same extension in $M$ as in $N$. Similarly, the case of
$\HOD$ and $\gHOD$ are definable transitive classes with a
definable well-order, and so by the theorem have the same
extension in $M$ and $N$. The proof method clearly applies
to many other definable classes.
\end{proof}

\begin{corollary}\label{Corollary.Noj:DefinableMofHOD^NtoN}
Suppose that $M\ofnoteq N$ are inner models of \ZF, and that
$M$ is definable in $N$ and $M\of \HOD^N$. Then there is no
nontrivial elementary embedding $j:M\to N$.
\end{corollary}

\begin{proof}
If there were such a $j:M\to N$, then by corollary \ref{Corollary.SameCardCofEtc.}, it follows that $\HOD^M=\HOD^N$. In
this case, it follows that $M\of\HOD^N=\HOD^M$ and so
$M=\HOD^M$, and consequently $N=\HOD^N$ and so $M=N$,
contrary to assumption.
\end{proof}

It follows immediately that there is no nontrivial
elementary embedding $j:\HOD^2\to\HOD$, if these models are
different, and indeed, there is no $j:\HOD^n\to\HOD^m$ for
any $m<n<\omega$, if the models are different. More
generally, if $\HOD^\eta$ exists and is definable in
$\HOD^\xi$ for some $\xi<\eta$, and
$\HOD^\eta\neq\HOD^\xi$, then there is no nontrivial
$j:\HOD^\eta\to\HOD^\xi$. Similarly, it follows immediately
from corollary \ref{Corollary.Noj:DefinableMofHOD^NtoN}
that there is no nontrivial $j:\gHOD^2\to \gHOD$ and no
nontrivial $j:\gHOD^n\to\gHOD^m$ for any $m<n<\omega$,
provided these models differ, and indeed, if $\gHOD^\eta$
exists and is definable in $\gHOD^\xi$ for some $\xi<\eta$,
then there is no nontrivial $j:\gHOD^\eta\to\gHOD^\xi$,
provided the models differ. In each case, we would have
definable $M\of\HOD^N$, and so corollary
\ref{Corollary.Noj:DefinableMofHOD^NtoN} rules out such
$j$.

Theorem \ref{Theorem.A^M=A^N} can be
applied to transitive {\it set} models $M$ and $N$, where
$M\of N$ is definable, they have the same ordinals and
there is a cofinal elementary embedding $j:M\to N$, and this situation
allows for several simplifications. In this case, one can
dispense with many of the metamathematical concerns about
uniform presentations, since one can perform the induction
in the ambient set theoretic background, where $j$ would
now be a set. One still seems to need that $M$ and $N$ are
well-founded in order to pick $x_0$ of minimal rank,
although this could also be possible even when the models
are ill-founded, as long as the $M^n$ are uniformly
presented in $N$ and $j$ is amenable to $N$, although such
a situation would be similar simply to applying the current
theorem inside $N$.

We also briefly note that the conclusion of theorem
\ref{Theorem.A^M=A^N} applies even in the case that the
class is defined using an ordinal parameter that is not
necessarily fixed by $j$.

\begin{corollary}\label{Corollary.B^M = B^N}
Suppose that $j:M\to N$ is an elementary embedding of inner
models $M\of N$ of \ZF, where $M$ is $b$-definable in $N$
with a parameter $b$ fixed by $j$. Suppose that
$B\of\HOD[b]^N$ is $(b,\beta)$-definable in $N$ for some
ordinal $\beta$. Then $B^M=B^N$.
\end{corollary}

\begin{proof}
Again, we do not need \AC\ here. Since $b$ and $\beta$ are
in both $M$ and $N$, it makes sense to consider $B^M$,
defined using the same definition as in $N$ and the same
parameters, so that $x\in B\iff\varphi(x,b,\beta)$ in
either model. In $N$, define the class
$A=\set{\<x,\alpha>\in\HOD[b]\st \varphi(x,b,\alpha)}$. This
class is $b$-definable in $N$ and contained in $\HOD[b]^N$. Thus, by theorem
\ref{Theorem.A^M=A^N}, it follows that $A^M=A^N$. But the
class $B$ is simply the $\beta^\th$ slice of $A$, and so it
follows that $B^M=B^N$.
\end{proof}

The previous corollaries will now allow us to apply the
stationary partition argument in a range of additional
situations. For example, we find the following consequence
striking, and it implies many of the other results we have
discussed.

\begin{theorem}\label{Theorem.Noj:definableMtoV}
If $M$ is a definable transitive class in $V$, then there
is no nontrivial elementary embedding $j:M\to V$.
\end{theorem}

\noindent This theorem should be understood as an \NGBC\
theorem scheme, asserting of each possible parameter-free definition of
such an $M$, that no \NGBC\ class is such an embedding $j$.
The theorem has theorem \ref{Theorem.NojHODtoV} as a
special case, asserting that there is no $j:\HOD\to V$,
simply because $\HOD$ is definable in $V$. But it
generalizes to show that there is no $j:\HOD^n\to V$ for
any natural number $n$, no $j:\gHOD\to V$ and no
$j:\gHOD^n\to V$, as all these classes are definable.
Theorem \ref{Theorem.Noj:definableMtoV} is itself an
immediate consequence of the following more general result,
simply by taking the case $N=V$.

\begin{theorem}\label{Theorem.Noj:definableMtoN}
Without assuming \AC\ in $V$, if $j:M\to N$ is a nontrivial
elementary embedding of inner models $M\of N$ of \ZFC\ and
$N$ is eventually stationary correct to $V$, then $M$ is
not definable in $N$ from parameters fixed by $j$.
\end{theorem}

\begin{proof}This theorem can be formalized as an \NGB\
scheme, asserting of every formula $\varphi$ that it is not
a definition of $M$ in $N$ by parameters fixed by $j$,
supposing that $j$, $M$ and $N$ are as described. Suppose
that $j:M\to N$ is a nontrivial elementary embedding of
inner models $M\of N$ of \ZFC, where $N$ is stationary
correct to $V$ above $\theta$, and suppose towards
contradiction that $M$ is $b$-definable in $N$ for some
parameter $b=j(b)$. Let $\kappa$ be the critical point of
$j$, which exists by lemma \ref{Lemma.CriticalPoint}, and
let $C=\set{\beta \st j \image \beta \of \beta}$ be the
class of ordinals closed under $j$, which is a closed
unbounded class of ordinals. Fix any $\eta\in C$ of
cofinality greater than $\kappa$ and $\theta$, and observe
that $C \intersect\eta$ is club in $\eta$. Let $S =
(\Cof_\omega \eta)^M$, which by corollary
\ref{Corollary.SameCardCofEtc.} is the same as
$(\Cof_\omega \eta)^N$, which is stationary in $N$ and
hence also in $V$. Thus, we may find $\lambda \in S
\intersect C$ above $\kappa$ and $\theta$. Since
$\cof(\lambda) = \omega$ in $M$ and $j \image \lambda \of
\lambda$, it follows that $j(\lambda) = \lambda$. Since $M$
and $N$ have the same cardinals by corollary
\ref{Corollary.SameCardCofEtc.}, and $N$ and $V$ have the
same cardinals above $\theta$ by the discussion about
stationary correctness in section
\ref{Section.StationaryCorrectNew}, it follows that
$(\lambda^+)^M=(\lambda^+)^N=(\lambda^+)^V$, a cardinal we
unambiguously denote $\lambda^+$, and so we conclude
$j(\lambda^+)=\lambda^+$.

Applying the Ulam-Solovay partition theorem in $M$, there is a
partition $\vec{S} = \< S_\alpha \st \alpha < \lambda^+
>\in M$ of ${(\Cof_\omega \lambda^+)}^M = {(\Cof_\omega
\lambda^+)}^N$ into disjoint sets stationary in $M$. Let
$\vec{T} = j(\vec{S})$. By elementarity and stationary
correctness, $T_\kappa$ is stationary in $V$. Since $C\cap
\lambda^+$ is club, there exists $\beta \in T_\kappa \cap
C$, and $\beta$ has cofinality $\omega$ in $N$ and hence in
$M$. It follows that $\beta \in S_\alpha$ for some $\alpha
< \lambda^+$, and that $j(\beta) = \beta$. Therefore,
$\beta \in T_{j(\alpha)} \neq T_\kappa$, contradicting the
disjointness of $\vec{T}$.
\end{proof}

As a quick example, we may immediately deduce the following
corollary, an extension of theorem
\ref{Theorem.Noj:definableMtoV} to the case of {\df
generic} embeddings, those that exist as classes in a
forcing extension. The corollary follows from theorem
\ref{Theorem.Noj:definableMtoN} with the observation that
$V$ is eventually stationary correct in all of its
set-forcing extensions.

\begin{corollary} \label{Corollary.NoGenericj:DefinableMtoV}
If $M$ is a definable class in $V$, then in no set-forcing
extension $V[G]$ is there a nontrivial elementary embedding $j:M\to V$.
\end{corollary}

For example, in any set forcing extension $V[G]$, there is
no nontrivial generic elementary embedding $j:\HOD\to V$,
no nontrivial generic embedding $j:\gHOD\to V$ and no
nontrivial generic embedding from the Mantle or from the
generic Mantle to $V$.

In connection with theorem \ref{Theorem.Noj:definableMtoV},
let us mention the very interesting work of Vickers and
Welch \cite[p. 1100]{VickersWelch2001:jMtoV}, who proved
that if $\ORD$ is Ramsey, then there is a nontrivial
elementary embedding $j: M \to V$, where $M$ is a
transitive inner model of \ZFC, and $j$ is definable from a
proper class that exists as the result of the large
cardinal assumption. Vickers and Welch note (p. 1090) that
$j$ cannot be definable in the usual sense, and our theorem
\ref{Theorem.Noj:definableMtoV} shows moreover that even
the class $M$ is not definable. On page 1101 of the same
paper, they reproduce a proof due to Foreman showing that
if $M$ is any inner model closed under $\omega$ sequences
of ordinals, then (assuming AC) there is no $j: M\to V$.
One very interesting part of this proof is the use of an
ultrapower construction to obtain a regular fixed point
above the critical point.

Let us deduce one final corollary of theorem
\ref{Theorem.A^M=A^N}, concerning the situation when \AC\
fails. As we shall mention later in questions
\ref{Question.nojVtoVwithoutAC?} and
\ref{Question.jHODtoHOD?}, it is not known whether one can
prove the Kunen inconsistency without the axiom of choice,
that is, whether there can be a nontrivial elementary
embedding $j:V\to V$ in the $\neg\AC$ context, nor whether
there can be a nontrivial elementary embedding
$j:\HOD\to\HOD$. Of course, the two questions are related,
because if there is $j:V\to V$ in a $\neg\AC$ context, then
the restriction of this embedding produces a nontrivial
elementary embedding $\HOD\to\HOD$. The next corollary
improves on the observation by showing that in order to
produce a nontrivial $j:\HOD\to\HOD$, it suffices to have a
nontrivial elementary embedding $j:M\to V$ in the $\neg\AC$
context from a definable $M$ to $V$.

\begin{corollary} \label{Corollary.jMtoVimpliesjHODtoHOD}
Do not assume \AC. If $j:M\to V$ is a nontrivial elementary
embedding from a transitive proper class $M$ that is
definable in $V$ from parameters fixed by $j$, then there
is a nontrivial elementary embedding from $\HOD$ to $\HOD$.
\end{corollary}

\begin{proof}By theorem \ref{Theorem.A^M=A^N}, it
follows from the assumption that $\HOD^M=\HOD$, and so
$j\restrict\HOD:\HOD\to\HOD$ is the desired embedding. Note
that lemma \ref{Lemma.CriticalPoint} shows that $j$ must
have a critical point, and so this restriction is
nontrivial.
\end{proof}

Although we do not know whether or not there can be
nontrivial $j:\HOD\to\HOD$, we merely note that corollary
\ref{Corollary.jMtoVimpliesjHODtoHOD} may be a way to
produce them.

\section{The case where $j$ is definable} \label{Section.jDefinable}

In this section, we consider the Kunen inconsistency and
its generalizations in the special case where the embedding
$j$ is not merely a class in \NGBC\ set theory, but a
first-order definable class. In this special case, many of
the results admit of particularly soft proofs, which we
shall presently describe, relying neither on any deep
combinatorial facts nor on the axiom of choice. These
results can be formalized as \ZF\ theorem schemes. Since
these soft proofs make essential use of the definability of
$j$, however, they do not seem to generalize to the
corresponding full results concerning embeddings that are
\NGBC\ classes and not necessarily definable from
parameters. In part for this reason, as we mentioned in
section \ref{Section.Preliminaries}, we find that the
\NGBC\ interpretation of the Kunen inconsistency seems to
provide it a more robust content than the \ZFC\ scheme
interpretation concerning only definable embeddings.
Nevertheless, because the soft proofs here do not use \AC,
we seem unable to deduce them directly from the prior
results, which do use \AC.

When we say that an embedding $j:M\to N$ on transitive
proper class models $M$ and $N$ of \ZF\ is definable in $V$
using parameter $p$, what we mean is that there has been
already provided a particular first-order formula
$\varphi(x,y,z)$, such that $j(x)=y$ if and only if
$\varphi(x,y,p)$ holds in $V$. That is, the relation
$\varphi(\cdot,\cdot,p)$ defines the graph of $j$. In
particular, the domain $M=\set{x\st \exists
y\,\varphi(x,y,p)}$ and the codomain $N=\Union\set{y\st
\exists x\,\varphi(x,y,p)}$ are also definable classes. (We
know that $N=\Union\set{y\st\exists x\,\varphi(x,y,p)}$
because $j$ is cofinal, and each element of $N$ is an
element of some $V_\alpha^N$.)

Conversely, we point out now that for a fixed first-order
formula $\varphi$, the question of whether a parameter $p$
succeeds in defining such an elementary embedding $j:M\to
N$ via $\varphi(\cdot,\cdot,p)$ is expressible as a
first-order property of $p$. To begin, it is easy to
express whether $\varphi(\cdot,\cdot,p)$ defines a
functional relation of its first two arguments. The
question of whether the domain $M$ and codomain $N$ of the
function are transitive proper class models of \ZF\ is
expressible by the method of lemma
\ref{Lemma.ZFFirstOrderExpressible}, and the question of
whether the function is elementary is expressible by the
method of lemma \ref{Lemma.Gaifman}. Similarly, when
$\varphi(\cdot,\cdot,p)$ does define an elementary
embedding, the question of whether this embedding is
nontrivial is easily expressible, as is the question of
whether the embedding has critical point $\kappa$. Putting
all this together, for a given formula $\varphi$ the
question whether a parameter $p$ succeeds in defining via
$\varphi(\cdot,\cdot,p)$ a nontrivial elementary embedding
$j:V\to V$ is a first-order expressible property of $p$.
Similarly, for a given formula $\varphi$, the collection of
ordinals $\kappa$ which arise as the critical point of a
nontrivial elementary embedding $j:V\to V$ defined by
$\varphi(\cdot,\cdot,p)$ for some parameter $p$ is a
definable class of ordinals.

These observations are all that are required now to prove
the Kunen inconsistency for embeddings that are definable
from parameters. This result, which had been part of the
folklore in some pockets of the set-theoretic community,
was published by Suzuki
\cite{Suzuki1999:NoDefinablejVtoVinZF}. The essence of the
proof is the classical observation that the concept of
being a Reinhardt cardinal, if consistent, cannot be first
order expressible, since if $\kappa$ is the least Reinhardt
cardinal, witnessed by $j:V\to V$, then by elementarity
$j(\kappa)$ would also be the least Reinhardt cardinal,
contrary to $\kappa<j(\kappa)$. Indeed, for the same
reason, there can be no consistent first-order property
$\varphi(\kappa)$ implying that $\kappa$ is Reinhardt.

\begin{theorem}[\cite{Suzuki1999:NoDefinablejVtoVinZF}]
Assume only \ZF. There is no nontrivial elementary
embedding $j:V\to V$ that is definable from
parameters.\label{Theorem.NoDefinablej:VtoVinZF}
\end{theorem}

\begin{proof} This is a theorem scheme, asserting of each
formula $\varphi$ that there is no parameter $p$ for which
$\varphi(\cdot,\cdot,p)$ defines an elementary embedding
$j:V\to V$. Suppose that for some parameter $p$, the
relation $\varphi(\cdot,\cdot,p)$ defines a nontrivial
elementary embedding $j:V\to V$. We may choose $p$ so that
the critical point $\kappa$ of this embedding is as small
as possible, among all parameters $w$ for which
$\varphi(\cdot,\cdot,w)$ defines a nontrivial elementary
embedding, since as we explained before the theorem, these
notions are first order expressible. In particular,
$\kappa$ is definable in $V$. Since $j:V\to V$ is
elementary, $j(\kappa)$ satisfies the same definition,
contradicting the fact that $\kappa<j(\kappa)$.
\end{proof}

The proof of theorem \ref{Theorem.NoDefinablej:VtoVinZF}
worked by observing that if $j:V\to V$ is definable in $V$,
even with parameters, then the concept of being Reinhardt
with respect to that definition for some parameter is first
order expressible. If it were consistent, then the least
such cardinal $\kappa$ would be definable, contrary to
$\kappa<j(\kappa)$. The argument therefore follows exactly
the pattern mentioned just before theorem
\ref{Theorem.NoDefinablej:VtoVinZF} of ruling out any
consistent first-order property implying the Reinhardt
property.

This method, of obtaining the smallest possible critical
point by quantifying over all possible choices of
parameter, thereby defining the cardinal without any
parameters, applies in many other contexts. We shall now
generalize the theorem by considering the possibility of
definable embeddings added by forcing.

\begin{theorem} \label{Theorem.DefinablejGeneralized}
Do not assume AC. If $M\of V[G]$ is a transitive inner
model of a set-forcing extension $V[G]$, then there is no
nontrivial elementary embedding $j:M\to V$, with a critical
point, that is definable in $V[G]$ from parameters.
\end{theorem}

\begin{proof}
Let $\Q$ be a forcing notion, $G \subseteq \Q$ generic.
Suppose there is such a $j:M\to V$ defined in $V[G]$ by
the formula $\varphi(\cdot,\cdot,b)$, using parameter $b\in
V[G]$, having critical point $\kappa$. Thus, there is some
 $\Q$-name $\dot b$ and some condition condition $q\in\Q$ which forces that
$\varphi(\cdot,\cdot,\dot b)$ defines a nontrivial
elementary embedding from a transitive inner model to
$\check V$ with critical point $\check\kappa$. Since this
property is first-order expressible, we may assume without
loss of generality that $\kappa$ is the smallest ordinal
for which there is a forcing notion $\Q$ with a condition
$q\in\Q$ and $\Q$-name $\dot b$ for which $q$ forces this
fact about $\check\kappa$. Thus, $\kappa$ is definable in
$V$ without parameters. Since $j:M\to V$ is elementary,
this implies that $\kappa$ must be in the range of $j$,
contrary to $\kappa$ being the critical point.
\end{proof}

The definable-embedding analogues of theorems
\ref{Theorem.NojVtoV}, \ref{Theorem.Noj:V[G]toV} and
\ref{Theorem.Noj:VtoV[G]} follow as immediate corollaries.
Statement \ref{Item.Suzuki's most general definable proof}
of corollary \ref{Corollary.DefinableGenericEmbeddings}
below is the most general result mentioned by Suzuki in
\cite[p. 1594]{Suzuki1999:NoDefinablejVtoVinZF} and was also noted
briefly by Vickers and Welch \cite[p.
1090]{VickersWelch2001:jMtoV}. Theorem
\ref{Theorem.DefinablejGeneralized} can be viewed as a
generic embedding analogue of it.

\begin{corollary}\label{Corollary.DefinableGenericEmbeddings}
Do not assume AC.
\begin{enumerate}
 \item (Suzuki) \label{Item.Suzuki's most general definable
     proof} For any transitive class $M$, there is no
     nontrivial elementary embedding $j:M\to V$, with a
     critical point, that is definable with parameters
     in $V$.
 \item There is no nontrivial elementary embedding
     $j:V\to V$ that is definable with parameters in a
     set-forcing extension $V[G]$.
 \item There is no nontrivial elementary embedding
     $j:V[G]\to V$ that is definable from parameters in
     such a forcing extension $V[G]$.
 \item There is no elementary embedding $j:V\to V[G]$,
     with a critical point, that is definable from
     parameters in such $V[G]$.
\end{enumerate}
\end{corollary}

\begin{proof}
This is a theorem scheme, asserting of each formula
$\varphi$ that it does not or is forced not to define such
an elementary embedding as stated in the relevant model.
Statement (1) is the special case of theorem
\ref{Theorem.DefinablejGeneralized} where the forcing was
trivial. In statements (2) and (3), we are guaranteed the
existence of a critical point by lemma
\ref{Lemma.CriticalPoint}, and so these are direct
instances of theorem \ref{Theorem.DefinablejGeneralized}.
Statement (4) follows from statement (1) applied in $V[G]$.
\end{proof}

Next, we turn to the impossibility of definable nontrivial
elementary embeddings from \HOD\ to \HOD.

\begin{theorem}\label{Theorem.NoDefinblej:HODtoHOD}
Do not assume \AC. There is no nontrivial elementary
embedding $j:\HOD\to\HOD$ that is definable in $V$ from
parameters.
\end{theorem}

\begin{proof} This is formally a \ZF\ theorem scheme, asserting of
each formula that it cannot define such an embedding.
Suppose that $j:\HOD\to\HOD$ is a nontrivial elementary
embedding from $\HOD$ to $\HOD$ defined from parameter $b$
by the formula $\varphi$, so that $j(x)=y$ if and only if
$V\satisfies\varphi(x,y,b)$. (We do not assume that $b$ is
in $\HOD$.) Let $\kappa$ be the critical point of $j$,
which exists by lemma \ref{Lemma.CriticalPoint}. Let
$\theta$ be the $\in$-rank of $b$, so that $b\in V_{\theta+1}$.
By the L\'{e}vy reflection theorem, there is an
ordinal-definable closed unbounded class $C$ of cardinals $\gamma$
above $\theta$ for which the formulas $\varphi$ and
$\exists y\varphi(x,y,z)$ are absolute between $V_\gamma$
and $V$. It follows from this absoluteness that
$j\image\gamma\of\gamma$ for any $\gamma\in C$. Let
$\delta$ be the $\omega^\th$ element of $C$ above $\kappa$
and $\theta$. In particular, $j\image\delta\of\delta$ and
$\HOD$ believes that $\delta$ has cofinality
$\omega$, which is less than $\kappa$. These two
facts imply that $j(\delta)=\delta$. From this, it follows
that $j((\delta^\plus)^{\HOD})=(\delta^\plus)^{\HOD}$.
Thus, we have found that $j$ has a fixed point above
$\kappa$ that is regular in $\HOD$. Now, let $\gamma$ be
the $(\delta^\plus)^{\HOD}$-th element of $C$ above
$\kappa$ and $\theta$. Thus, $\gamma\in C$ and so
$j\image\gamma\of\gamma$. Since $C$ is definable in $V$
without parameters, it follows that every initial segment
of $C$ is in $\HOD$, and so $\HOD$ can see that
$\cof(\gamma)=(\delta^\plus)^{\HOD}$. And since
$(\delta^\plus)^{\HOD}$ is fixed by $j$, it follows that
$j(\gamma)=\gamma$.

The point now is that this is enough to run the main
stationary-partition argument, as in theorem
\ref{Theorem.Noj:V[G]toV}. Namely, by the Ulam-Solovay theorem in
$\HOD$, there is a partition of
$(\Cof_\omega\gamma)^\HOD=\bigsqcup_{\alpha<\kappa}S_\alpha$
into stationary sets $\vec
S=\<S_\alpha\st\alpha<\kappa>\in\HOD$. Let $\vec T=j(\vec
S)=\<T_\alpha\st\alpha<j(\kappa)>$, and let $S^*=T_\kappa$,
which is a stationary subset of $(\Cof_\omega\gamma)^\HOD$
in $\HOD$. Since $C\intersect\gamma\in\HOD$ and
$C\intersect\gamma$ is closed and unbounded in $\gamma$,
there is $\beta\in S^*\intersect C$. Since $\beta\in S^*$,
it follows that $\beta$ has cofinality $\omega$ in $\HOD$
and consequently $\beta\in S_\alpha$ for some
$\alpha<\kappa$. Since $\beta\in C$, however, it follows
also that $j\image\beta\of\beta$ and consequently
$j(\beta)=\beta$, which means that $\beta=j(\beta)\in
j(S_\alpha)=T_\alpha$. Thus, $\beta$ is in both $T_\alpha$
and $T_\kappa$, contradicting the fact that these are
disjoint, being distinct elements of the partition $\vec
T$.
\end{proof}

Theorem \ref{Theorem.NoDefinblej:HODtoHOD} is generalized
by theorem \ref{Theorem.NojMtoNwithHODofNomegaclosed}
below, which uses a different method. Note that we could
have used the reflection argument involving $C$ of this
proof in several of the earlier cases, where we had wanted
to find a club of closure points of $j$. In general, when
$j$ is merely an \NGB\ class, not necessarily definable, we
can still apply the Reflection theorem in $\ZFC(j)$ to find
a closed unbounded class $C$ of cardinals $\gamma$
reflecting a given finite collection of statements with
class parameter $j$, and $C$ will be definable from $j$.

The proof of theorem \ref{Theorem.NoDefinblej:HODtoHOD} is
much simpler in the case where $j$ is definable without
parameters or with ordinal parameters, for in this case one
gets directly that $j\restrict\theta\in\HOD$ for every
ordinal $\theta$, and this is enough to complete the
argument. Indeed, when $j:\HOD\to\HOD$ is definable in $V$
using no parameters or using ordinal parameters, then
$\HOD$ satisfies $\ZFC(j)$ and so we have directly an
instance of the Kunen inconsistency by restricting to
$\<\HOD,\in,j>$.

Since theorems \ref{Theorem.Noj:VtoHOD} and
\ref{Theorem.NojHODtoV} already rule out all class
embeddings of the form $V\to\HOD$ or $\HOD\to V$, without
using \AC, there is no need to consider definable
embeddings of that form.

Note that we made the assumption in theorem
\ref{Theorem.DefinablejGeneralized} and also in corollary
\ref{Corollary.DefinableGenericEmbeddings} that the
embedding $j$ has a critical point, because the situation
of these results are not directed covered by lemma
\ref{Lemma.CriticalPoint}. It is however conceivable to us
that a strengthened version of lemma
\ref{Lemma.CriticalPoint} might be possible, showing that
all such embeddings must have a critical point and allowing
us to eliminate that assumption.

%
Let us conclude with a generalization of theorem
\ref{Theorem.NoDefinblej:HODtoHOD}. Using the notation
$[A]^\omega$ to denote the set of subsets of $A$ of order
type $\omega$, where $A$ is a set of ordinals, we shall
appeal to the Erd\H{o}s-Hajnal theorem (see
\cite{Kanamori2004:TheHigherInfinite2ed}) in the following
form: If $\lambda$ is any ordinal, then there is an
$\omega$-\Jonsson\ function for $\lambda$, that is, a
function $f:[\lambda]^\omega\to\lambda$ such that for any
$A\of\lambda$ of size $\lambda$, then $f\image
[A]^\omega=\lambda$. Let $\ZFC_{+2}$ be the theory
consisting of the sentences \ZFC\ proves to be true in
$V_{\lambda+2}$, whenever $\lambda$ is a limit ordinal. For
example, one of these sentences is the existence of an
$\omega$-\Jonsson\  function
$f:[\lambda]^\omega\to\lambda$, since such a function
exists in $V_{\lambda+2}$, if one uses a flat pairing
function.

\begin{lemma}\label{Lemma.jMtoNofZFC+2}
Suppose that $j:M\to N$ is an elementary embedding of
transitive models $M,N\satisfies\ZFC_{+2}$, both of height
$\lambda+2$, and that $([\lambda]^\omega)^N\of M$ and
$\kappa<j(\kappa)$ for some $\kappa<\lambda$. Then there is
no set $T\of\lambda$ of size $\lambda$ in $N$ such that
$j(\beta)=\beta$ for all $\beta\in T$.
\end{lemma}

\begin{proof}
Suppose that $j:M\to N$ is elementary and that $M$ and $N$
are transitive models of $\ZFC_{+2}$ of height $\lambda+2$
and that $\kappa<\lambda$ is the critical point of $j$.
Note that $j(\lambda)=\lambda$, since this is the second
largest ordinal of both $M$ and $N$. Suppose toward
contradiction that there is a set $T\of\lambda$ of size
$\lambda$ in $N$ consisting of fixed points of $j$. By the
\Erdos-Hajnal theorem, there is an $\omega$-\Jonsson\
function $f:[\lambda]^\omega\to\lambda$ in $M$. By
elementarity, $j(f)$ is an $\omega$-\Jonsson\ function
$j(f):[\lambda]^\omega\to\lambda$ in $N$. Since
$T\of\lambda$ has size $\lambda$, it follows by the
$\omega$-\Jonsson\ property that there is some
$s\in[T]^\omega$ in $N$ such that $j(f)(s)=\kappa$, since
the map is onto $\lambda$. Our assumption that
$([\lambda]^\omega)^N\of M$ implies that $s\in M$. But since
$s$ has order type $\omega$ and $j(\beta)=\beta$ for each
$\beta\in T$, it follows that $j(s)=j\image s=s$. Thus,
$\kappa=j(f)(s)=j(f)(j(s))=j(f(s))$, which would place
$\kappa$ in the range of $j$, which is impossible for the
critical point.
\end{proof}

\begin{theorem}\label{Theorem.NojMtoNwithHODofNomegaclosed}
Do not assume \AC. Suppose that $M$ and $N$ are transitive
proper class models of \ZFC\ with $\HOD\of N$ and
$([\ORD]^\omega)^N\of M$. Then there is no nontrivial
elementary embedding $j:M\to N$ that is definable in $V$
from parameters.
\end{theorem}

\begin{proof}
Suppose that $j$ is defined by the formula $\varphi$ and
parameter $s$, so that $j(x)=y\iff
V\satisfies\varphi(x,y,s)$. Let $\kappa$ be the critical
point of $j$, which exists by lemma
\ref{Lemma.CriticalPoint}. By the \Levy\ reflection
theorem, there is an ordinal-definable closed unbounded
class $C$ of cardinals $\delta$ above $\kappa$ and the rank
of $s$ such that $\varphi$ and $\exists y\varphi$ both
reflect from $V$ to $V_\delta$. It follows as in the proof
of theorem \ref{Theorem.NoDefinblej:HODtoHOD} that every
$\delta\in C$ has $j\image\delta\of\delta$. Notice that if
$\delta$ is the $\omega^\th$ element of $C$ above any
ordinal---these are exactly the {\df simple limits} of $C$,
that is, limit points of $C$ that are not limits of limits
of $C$---then $\delta$ will have cofinality $\omega$ in
$\HOD$ and hence also in $N$ and hence in $M$, from which
it follows that $j(\delta)=\delta$. Let $\beta_0$ be the
first simple limit of $C$; let $\beta_{n+1}$ be the
$\beta_n^\th$ simple limit of $C$; and let
$\lambda=\sup_n\beta_n$. It follows that $\lambda\in C$ and
that $\lambda$ has cofinality $\omega$ in $\HOD$ and hence
in $N$ and hence in $M$, and this implies
$j(\lambda)=\lambda$. But also, the set of simple limits of
$C$ below $\lambda$ has size $\lambda$ and is in \HOD\ and
hence in $N$. Since these are all fixed points of $j$, the
restriction $j\restrict V_{\lambda+2}^M\to V_{\lambda+2}^N$
violates lemma \ref{Lemma.jMtoNofZFC+2}, a contradiction.
\end{proof}

%
%

\section{Open Questions}

We conclude this article by mentioning the two most
prominent open questions remaining in this area. Perhaps
the principal open question is whether one can prove the
Kunen inconsistency without using the axiom of choice.

\begin{question} Is it consistent without \AC\ that $j:V\to V$ is a
nontrivial elementary embedding of the universe to
itself?\label{Question.nojVtoVwithoutAC?}
\end{question}

We are also naturally interested in the corresponding
question for each of the generalizations of the Kunen
inconsistency whose current proofs use \AC. For example, in
the $\neg\AC$ context can there be nontrivial elementary
embeddings $j:V[G]\to V$ or $j:V\to V[G]$ for a set-forcing
extension $V[G]$?

Of course, theorem \ref{Theorem.NoDefinablej:VtoVinZF} and
the other results of section \ref{Section.jDefinable}
settle the case of definable embeddings $j$ in \ZF. In
particular, if one understands the Kunen inconsistency
solely as a \ZF\ scheme, in the manner we described in
section \ref{Section.Preliminaries}, then one might regard
question \ref{Question.nojVtoVwithoutAC?} as settled by
theorem \ref{Theorem.NoDefinablej:VtoVinZF}. But we ask the
question in the context of \NGB, where the embedding $j$
will be merely a class in \NGB, not necessarily first-order
definable from parameters. With this interpretation, the
theorems here on definable embeddings do not answer it. An
equivalent formalization of the question would inquire
whether it is consistent with $\ZF(j)$, which allows
formulas using the class $j$ into the Separation and
Replacement schemes, that $j:V\to V$ is a nontrivial
elementary embedding of the universe to itself. All of the
arguments we have used in this article to refute the
existence of such $j$ have either used the stationary partition theorem, which relies on \AC, or have made
assumptions on the definability of $j$. Thus, they do not
answer question \ref{Question.nojVtoVwithoutAC?}.

%
%

A second, related open question is whether there can be a
nontrivial elementary embedding from $\HOD$ to $\HOD$.

\begin{question} \label{Question.jHODtoHOD?}
Is it consistent that there is a nontrivial elementary
embedding $j:\HOD\to \HOD$?
\end{question}

We ask the question in the \NGBC\ or $\ZFC(j)$ contexts,
although it is also sensible to drop \AC\ here. Theorem
\ref{Theorem.NoDefinblej:HODtoHOD} refutes without \AC\
such $j$ that are definable from parameters. Of course, if
$j:V\to V$ is a nontrivial elementary embedding, then so is
$j\restrict\HOD:\HOD\to\HOD$, and so an affirmative answer
to question \ref{Question.nojVtoVwithoutAC?} would imply anf
affirmative answer to question \ref{Question.jHODtoHOD?} in the $\neg\AC$ context. A
negative answer to question \ref{Question.jHODtoHOD?} in
the $\neg\AC$ context, in addition to implying a negative
answer to question \ref{Question.nojVtoVwithoutAC?}, would
also imply by corollary
\ref{Corollary.jMtoVimpliesjHODtoHOD} that there is no
nontrivial elementary embedding $j:M \to V$ whenever $M$ is
definable without parameters or with parameters fixed by
$j$.

In addition to these two questions, of course, there are
numerous others. For example, to what extent do the
theorems we have mentioned about embeddings arising in
set-forcing extensions also apply to class forcing? Or to
certain kinds of class forcing? Or to other non-forcing
extensions? To what extent do the theorems on $\HOD$
generalize to other natural definable classes? We should
like to know the answers.

\bibliographystyle{alpha}
\bibliography{HamkinsBiblio,MathBiblio}
\end{document}